\documentclass[12pt,reqno]{amsart}
\usepackage{mathrsfs}\overfullrule=5pt
\usepackage{amsfonts,epsfig,amsmath,amssymb,amscd,graphicx}
\usepackage[colorlinks=true]{hyperref}
\hypersetup{urlcolor=red, citecolor=blue}
\usepackage{pstricks}
\usepackage[titletoc]{appendix}
\input xy
\xyoption{all}
\allowdisplaybreaks[4]
\newtheorem{thm}{Theorem}[section]
\newtheorem{lem}[thm]{Lemma}

\newtheorem{prop}[thm]{Proposition}

\newtheorem{cor}[thm]{Corollary}

\theoremstyle{definition}

\theoremstyle{remark}
\newtheorem{rem}[thm]{Remark}

\numberwithin{equation}{section}

\def \N {\mathbb N}
\def \R {\mathbb R}

\begin{document}

	\title[Rel-u.p.e of induced amenable group actions]{Relative uniformly positive entropy of induced amenable group actions}
	\author{Kairan Liu}
	\address{Kairan Liu,
		\small{College of Mathematics and Statistics, Chongqing University, 
			Chongqing 401331, P.R. China}}
	\email{lkr111@cqu.edu.cn}
	
	\author{Runju Wei}
	\address{Runju Wei,
		Department of Mathematics, University of Science and Technology of China, Hefei, Anhui 230026, China}
	\email{wrj3219@mail.ustc.edu.cn}
	
	\date{\today}

	\begin{abstract}
		Let $G$ be a countably infinite discrete amenable group.
		It should be noted that a $G$-system $(X,G)$ naturally induces a $G$-system $(\mathcal{M}(X),G)$, where 
		$\mathcal{M}(X)$ denotes the space of Borel probability measures on the compact metric space $X$ endowed 
		with the weak*-topology. A factor map $\pi\colon (X,G)\to(Y,G)$ between two $G$-systems induces a factor map 
		$\widetilde{\pi}\colon(\mathcal{M}(X),G)\to(\mathcal{M}(Y),G)$. It turns out that $\widetilde{\pi}$ is open if and only if $\pi$ is open. 
		When $Y$ is fully supported, it is shown that $\pi$ has relative uniformly positive entropy if and only if $\widetilde{\pi}$ has 
		relative uniformly positive entropy.
	\end{abstract}
	
	\maketitle 
	
	\section{Introduction}
	In the process of studying the classification of topological dynamical systems, entropy as a conjugacy invariant 
	plays an important role which divides them into two classes. For $\mathbb{Z}$-systems, the notion of uniformly 
	positive entropy (u.p.e for short) was introduced by Blanchard in \cite{B} as an analogue in topological 
	dynamics for the notion of a K-process in ergodic theory. He then naturally defined the notion of entropy pairs 
	and used it to show that a u.p.e system is disjoint from all minimal zero entropy systems \cite{Blanchard2}. 
	Further research concerning u.p.e systems and entropy pairs can be found in \cite{Blanchard-Host,
		Blanchard-Glasner 1997,Glasner-Weiss 1994,L-S,Huang-Ye,Glasner-Ye}.
	
	Recently, there are a lot of significant progress in studying relative entropy via local relative entropy
	theory for $\mathbb{Z}$-systems. For a factor map between two $\mathbb{Z}$-systems, Glasner and Weiss 
	\cite{Glasner-Weiss-2} introduced the relative uniformly positive entropy (rel-u.p.e) and the notion of relative topological Pinsker factor based on the idea of u.p.e. 
	extensions. Later, Park and Siemaszko \cite{P-S} interpreted another relative topological Pinsker factor, defined by Lemanczyk and 
	Siemaszko \cite{L-S}, using relative measure-theoretical entropy and discussed relative product. In 
	\cite{Huang-Ye-Zhang-1}, Huang, Ye and Zhang introduced the notions of relative entropy tuples in both topological and 
	measure-theoretical settings. They showed
	that the finite product of rel-u.p.e extensions has
	rel-u.p.e iff the factors are fully supported (definitions see Section \ref{desupp}). They
	also proved some classical results about rel-u.p.e extension. 
	We will refer redears to \cite{B-F-F,D-S,Huang-Ye-Zhang2006,L-W} for more results related to local relative entropy
	theory.
	
	
	Bauer and Sigmund \cite{Bauer-Sigmund} initiated a systematic study of the connections between dynamical properties
	of $\mathbb{Z}$-system and its induced system (whose phase space consists of all Borel probability measures on 
	the original space, for details see Section \ref{sec:pre}). A well-known result due to Glasner and
	Weiss \cite{Glasner-Weiss-1} in 1995 reveals that if a system has zero topological entropy, then so does
	its induced system. Later, this connection was further developed by Kerr and Li in \cite{Kerr-Li-2}.
	They obtained that a system is null if and only if its induced system is null. More research concerning 
	relations of these systems was developed in \cite{Akin-Auslander-Nagar2017,Banks,Sharma-Nagar,Zhou-Qiao}. Recently, Bernardes et al 
	\cite{N-U-A} proved that a $\mathbb{Z}$-system has u.p.e if and only if so does its induced
	system.  
	
	After Ornstein and Weiss's pioneering work for amenable group actions in 1987 
	\cite{O-W-2}, there are a lot of results developed in the process of studying the amenable group action systems. 
	We will refer the redear to see the related papers \cite{Lindenstrauss,R-W,W-Z,Weiss,Zimmer,Huang-Ye-Zhang-2}. 
	In this paper,
	we always assume that $G$ is a countably infinite discrete amenable group. By a \textbf{$G$-system} $(X,G)$ we mean a compact metric space
	$X$ together with $G$ acting on $X$ by homeomorphisms, that is, there exists a continuous map 
	$\Gamma:G\times X\to X$, satisfying 
	\begin{itemize}
		\item{$\Gamma(e_{G},x)=x$ for every $x\in X$.}
		\item{$\Gamma(g,\Gamma(h,x))=\Gamma(gh,x)$ for each $g,h\in G$ and $x\in X$.}
	\end{itemize}
	We write $\Gamma(g,x)$ as $gx$ for every $g\in G$ and $x\in X$.
	
	Motivated by those works which were previously mentioned for $\mathbb{Z}$-systems and 
	the local entropy theory developed for countable discrete amenable group action systems due to 
	Huang, Ye and Zhang \cite{Huang-Ye-Zhang-2}, Kerr and Li \cite{Kerr-Li-1}, the present paper aims to investigate the properties of relative uniformly positive 
	entropy (rel-u.p.e) for induced factor map of a factor map between two
	$G$-systems (see Section \ref{sec:pre} for definitions). 
	
	More precisely, let $(X,G)$ be a $G$-system, $\mathcal{B}_X$ be the set of Borel subsets of $X$ and $\mathcal{M}(X)$
	be the space of Borel probability measures on the compact metric space $X$
	endowed with the weak*-topology. 
	Then $G$-system $(X,G)$ induces a system $(\mathcal{M}(X),G)$ (see Section \ref{sec:pre} for details). 
	For any $x\in X$, let $\delta_{x}$ denote the Dirac measure on $x$ and 
	$$\mathcal{M}_n(X)=\{\frac{1}{n}\sum_{i=1}^n\delta_{x_i}: x_1, x_2,\cdots,x_n\in X \}$$
	for each $n\in \mathbb{N}$. Then $\mathcal{M}_n(X)$ is closed and invariant under $G$ (i.e. $g\mathcal{M}_n(X)=\mathcal{M}_n(X)$ for every
	$g\in G$). Hence we can consider the subsystems $(\mathcal{M}_n(X),G)$ of $(\mathcal{M}(X),G)$ for each 
	$n\in \mathbb{N}$. For a factor map $\pi:(X,G)\to (Y,G)$ between two
	$G$-systems, when $supp(Y)=Y$ (definitions see Section \ref{desupp}) we have the following result. 

	\begin{thm}\label{main} Let $\pi\colon (X,G)\to (Y,G)$ be a factor map between two $G$-systems,	
	$\widetilde{\pi}\colon (\mathcal{M}(X),G)\to(\mathcal{M}(Y),G)$ be the factor map induced by $\pi$ and 
	$\widetilde{\pi}_n\colon (\mathcal{M}_n(X),G)$ $\to(\mathcal{M}_n(Y),G)$ be the restriction of $\widetilde{\pi}$
	on $\mathcal{M}_n(X)$. When $supp(Y)=Y$, the following are equivalent
	\begin{itemize}
  \item[(1)] $\pi$ has relative uniformly positive entropy;
  \item[(2)]  $\widetilde{\pi}_n$ has relative uniformly positive entropy for some $n\in\mathbb{N}$;
  \item[(3)] $\widetilde{\pi}_n$ has relative uniformly positive entropy for every $n\in\mathbb{N}$;
  \item[(4)]$\widetilde{\pi}$ has relative uniformly positive entropy. 
	\end{itemize}
	\end{thm}
	Notice that when $Y$ is a singleton, we obtain that $(X,G)$ has u.p.e if and only if the induced system 
	$(\mathcal{M}(X),G)$ has u.p.e (when $G=\mathbb{Z}$, see \cite[Theorem 4]{N-U-A}).
	
	We say a map $\pi\colon X\to Y$ between two topological spaces is \textbf{open} if the images of open sets are open. 
	Then we have the following result.
	\begin{thm}\label{prop-1.4}Let $\pi:X\to Y$ be a  surjective continuous map between two compact metrizable spaces, and
		$\widetilde{\pi}:\mathcal{M}(X)\to \mathcal{M}(Y)$ be the induced map of $\pi$. 
		Then $\pi$ is open if and only if $\widetilde{\pi}$ is open.
	\end{thm}

	This paper is organized as follows. In Section \ref{sec:pre}, we will list some basic notions and results needed in our 
	argument. In Section \ref{section 3} and Section \ref{section 4}, we will give a proof of Theorem \ref{main}.
	Finally, we prove Theorem \ref{prop-1.4} in Section \ref{section 5}.
	
	\textbf{Acknowledgements.}
	The authors would like to thank Prof. Wen Huang, Prof. Hanfeng Li, Dr. Yixiao Qiao and Dr. Lei Jin for their useful comments and 
	suggestions. Kairan Liu is supported by China Postdoctoral Science Foundation (No. 2022M710527).
	
	\section{Preliminaries}\label{sec:pre}
	In this section, we recall some basic notations and results which will be used repeatly in our paper. Denote by $\mathbb{N}$ and $\mathbb{R}$ the set of natural numbers and real numbers respectively.
	For $n\in\mathbb{N}$, we write $[n]$ for $\{1,2, \cdots,n\}$.
    \subsection{Amenable group}
    We say a countably infinite discrete group $G$ is \textbf{amenable} if there always exists an invariant Borel probability measure when it acts on any compact
	metric space. In the case $G$ is a countably infinite discrete group, amenability is equivalent to the existence of a 
	\textbf{F$\phi$lner sequence}: a 
	sequence of nonempty  finite subsets $\{F_n\}_{n=1}^{\infty}$ of $G$ such that
	$$\lim_{n\to\infty}\frac{\vert F_n\Delta gF_n\vert}{\vert F_n\vert}=0$$
    for all $g\in G$. One can see Ornstein and Weiss' paper \cite{O-W-2} for more details about amenable group.
	In this paper, we always assume that $G$ is a countably infinite discrete amenable group and denote
	by $\mathcal{F}(G)$ the collection of nonempty finite subsets of $G$. 
	The following result is well-known (see \cite[Theorem 4.48]{Kerr-Li-2016}).
	\begin{thm}\label{amenable converges}Let $\phi$ be a real-valued function on $\mathcal{F}(G)$
		satisfying
		\begin{itemize}
			\item[(1)]$\phi(Fs)=\phi(F)$ for all $F\in\mathcal{F}(G)$ and $s\in G$, and 
            \item[(2)]$\phi(F)\leq\frac{1}{k}\sum_{E\in\mathcal{E}}\phi(E)$ for every $k\in\mathbb{N}$, $F\in\mathcal{F}(G)$
			 and finite collection $\mathcal{E}\subseteq\mathcal{F}(G)$ with $\bigcup_{E\in\mathcal{E}}E\subseteq F$ and 
			$\sum_{E\in\mathcal{E}}1_E\geq k1_F$.
		\end{itemize}
		Then $\frac{\phi(F)}{\vert F\vert}$ converges to a limit as $F$ becomes more and more invariant and this limit is
		equal to $\inf_{F}\frac{\phi(F)}{\vert F\vert}$,
		where $F$ ranges over all nonempty finite subsets of $G$.
	\end{thm}

	\subsection{Induced systems}
	Assume that $X$ is a compact metric space. Let $\mathcal{B}_X$ be the collection of Borel subsets of $X$, $C(X)$ be the space of continuous 
	maps from $X$ to $\R$ endowed with the supremum norm $||\cdot||_{\infty}$, and $\mathcal{M}(X)$ be the set of Borel 
	probability measures on $X$ endowed with the \textbf{weak$^\ast$-topology}, which is the smallest topology making the map
	$$D_g:\mathcal{M}(X)\to\mathbb{R},\;\;\;\mu\mapsto\int_Xgd\mu$$ 
	continuous for every $g\in C(X)$, and the topology basis of weak$^\ast$-topology consists of the following sets
	\begin{small}
		\begin{align}\label{basis of M(X)-function}
		\mathbb{V}(\mu;f_1,\cdots,f_k;\epsilon)\colon=
		\bigg\{\nu\in\mathcal{M}(X): \big\vert\int_Xf_id\mu-\int_Xf_id\nu\big\vert<\epsilon, \ \text{for\ all}\ i\in[k]\bigg\},
		\end{align}
		\end{small}
   where $\mu\in\mathcal{M}(X)$, $k\geq 1$, $\epsilon>0$ and $f_i:X\to\mathbb{R}$ are continuous functions for $i\in[k]$. The \textbf{Prohorov metric} on $\mathcal{M}(X)$
	$$\begin{small}d_P(\mu,\nu):=inf\big\{\delta>0:\mu(A)\leq\nu(A^{\delta})+\delta\ 
		\text{and}\ \nu(A)\leq \mu(A^{\delta})+\delta\ \text{for\ all}\ A\in\mathcal{B}_X\big\},\end{small}$$
	where $A^{\delta}=\{x\in X:\ d(x,A)<\delta\}$, is compatible with the weak*-topology. We will refer the readers 
	to the books \cite{P,RMD,ASK} for the knowledge of space $\mathcal{M}(X)$. Moreover, 
	$$d_P(\mu,\nu)=
	inf\big\{\delta>0:\mu(A)\leq\nu(A^{\delta})+\delta\ \text{for\ all}\ A\in\mathcal{B}_X\big\}$$ 
	(see \cite[Page 72]{P}). The following Proposition \ref{basis of M(X)} describes a basis for the weak*-topology on 
	$\mathcal{M}(X)$ due to Bernardes et al (see \cite[Lemma 1]{N-U-A} ).
	\begin{prop}\label{basis of M(X)}
		The set of the form
		$$\mathbb{W}(U_1,U_2,\cdots,U_k:\eta_1,\eta_2,\cdots,\eta_k):=\big\{\mu\in\mathcal{M}(X)\colon \ \mu(U_i)>\eta_i\ \text{for}\ 
		i\in[k]\big\},$$
		where $k\geq1$, $U_1,U_2,\cdots,U_k$ are nonempty disjoint open sets in $X$ and $\eta_1,\eta_2,\cdots,\eta_k$ 
		are positive real numbers with $\eta_1+\eta_2+\cdots+\eta_k<1$, form a basis for the weak*-topology on 
		$\mathcal{M}(X)$.
	\end{prop}
	
	A $G$-system $(X,G)$ induces a system $(\mathcal{M}(X),G)$, where $g\colon \mathcal{M}(X)\to\mathcal{M}(X)$ defined by
	$(g\mu)(A)\colon=\mu(g^{-1}A)$, for every $g\in G$, $\mu\in\mathcal{M}(X)$ and $A\in\mathcal{B}_X$.
	We call $(\mathcal{M}(X),G)$ the \textbf{induced system} of $(X,G)$.

	Let $(X,G)$ and $(Y,G)$ be two $G$-systems. A continuous map $\pi:(X,G)\to(Y,G)$ is called a \textbf{factor map} 
	between $(X,G)$ and $(Y,G)$ if it is onto and $\pi\circ g=g\circ\pi$ for every $g\in G$. $\pi$ can induce a factor map 
	$\widetilde{\pi}:(\mathcal{M}(X),G)\to(\mathcal{M}(Y),G)$ by
	$$(\widetilde{\pi}\mu)(B)=\mu(\pi^{-1}B)$$ 
	for every $\mu\in\mathcal{M}(X)$ and $B\in\mathcal{B}_Y$. For every $n\in\mathbb{N}$, we denote 
	$$\widetilde{\pi}_n\colon=\tilde{\pi}\vert_{\mathcal{M}_n(X)}\colon \mathcal{M}_n(X)\to\mathcal{M}_n(Y)$$ 
	by the restriction of $\tilde{\pi}$ on $\mathcal{M}_n(X)$. Note that $\widetilde{\pi}_n$ is also a factor map for each $n\in\mathbb{N}$.
	\subsection{Support}\label{desupp} Let $(X,G)$ be a $G$-system, $(\mathcal{M}(X),G)$ be the induced $G$-system of $(X,G)$.
	We denote by $\mathcal{M}(X,G)$ the set of all $G$-invariant measures.
	For $\mu\in\mathcal{M}(X)$, we
	denote by $supp(\mu)$ the \textbf{support of $\mu$}, i.e. the smallest closed subset $W\subseteq X$ such that
	$\mu(W)=1$. We denote by $supp(X,G)$ the \textbf{support of $(X, G)$}, i.e. 
	$$supp(X, G) =\bigcup_{\mu\in\mathcal{M}(X,G)}supp(\mu).$$
	$(X,G)$ is called \textbf{fully supported} if there is an invariant measure $\mu\in\mathcal{M}(X,G)$
	with full support (i.e. $supp(\mu)=X$), equivalently, $supp(X, G) = X$.

	\subsection{Relative uniformly positive topological entropy}
	For a given $G$-system $(X,G)$, a \textbf{cover} of $X$ is a family of Borel subsets of $X$, whose union is $X$. 
	Denote the set of finite covers by $\mathcal{C}_X$. For $n\in\N$ and $\mathcal{U}_1,\mathcal{U}_2,\dots,\mathcal{U}_n\in \mathcal{C}_X$, we denote
	$$\bigvee_{i=1}^{n}\mathcal{U}_i=\bigg\{A_1\cap A_1\cap\dots\cap A_n: A_i\in\mathcal{U}_i,\ i\in[n]\bigg\}.$$
	
	Let $\pi:(X,G)\to (Y,G)$ be a factor map between two $G$-systems and $\mathcal{U}\in\mathcal{C}_X$.
	For any non-empty subset $E$ of $X$, let $N(\mathcal{U},E)$ be the minimum among the cardinalities of the subsets of 
	$\mathcal{U}$ which cover $E$, and define 
	$$N(\mathcal{U}|\pi)=\sup_{y\in Y}N(\mathcal{U},\pi^{-1}(y)).$$
	The \textbf{topological conditional entropy of $\mathcal{U}$ with respect to $\pi$} is defined by
	$$h_{top}(\mathcal{U},G|\pi)=\lim_{n\to\infty}\frac{1}{\vert F_n\vert}\log N(\mathcal{U}_{F_n}|\pi),$$
	where $\mathcal{U}_{F_{n}}=\bigvee_{g\in F_{n}}g^{-1}\mathcal{U}$
	and  $\{F_n\}_{n=1}^{\infty}$ is a F$\phi$lner sequence of $G$. It is well known that $h_{top}(\mathcal{U},G\vert\pi)$ is well-defined and is independent of the choice 
	of the F$\phi$lner sequences of $G$. 
	
	Let $\pi:(X,G)\to(Y,G)$ be a factor map between $G$-systems. $\mathcal{U}=\{
	U_{1},\cdots,U_{n}\}\in\mathcal{C}_X$ is said to be \textbf{non-dense-on-$\pi$-fiber} if there is $y\in Y$ such that 
	$\pi^{-1}(y)$ is not contained in any element of $\overline{\mathcal{U}}$ which consists of the closures of 
	elements of $\mathcal{U}$ in $X$. Clearly, if an open cover $\mathcal{U}=\{U_{1},U_{2}\}$ is non-dense-on-$\pi$-fiber, then $\pi(U_{1})\cap\pi(U_{2})\neq\emptyset$. 
	We say $(X,G)$ or $\pi$ has \textbf{relative uniformly positive entropy} (rel-u.p.e for short), if for any
	non-dense-on-$\pi$-fiber open cover $\mathcal{U}$ of $X$ with two elements, we have $h_{top}(\mathcal{U},G|\pi)>0$.
	
	For $n\in\mathbb{N}$ and $G$-systems $(Z_i,G)$, $i\in[n]$, we set
	$$\prod_{i\in[n]}Z_i=\big\{(z_1,z_2,\cdots,z_n)\colon\ z_i\in Z_i;\  i\in[n]\big\}$$
	and
	$$g(z_1,z_2,\cdots,z_n)=(gz_1,gz_2,\cdots,gz_n)$$
	for every $g\in G$ and $z_i\in Z_i$ for $i\in[n]$. Clearly, $(\prod_{i\in[n]}Z_i,G)$ is also a $G$-system.
	When $Z_i=Z$ for all $i\in[n]$, we write $\prod_{i\in[n]}Z_i$ as $Z^{(n)}$. Let $\pi_i:(X_i,G)\to (Y_i,G)$ be 
	factor maps between $G$-systems for $i\in[n]$. Then $\{\pi_i\}_{i\in[n]}$ induce a factor map 
	$$\prod_{i\in[n]}\pi_i\colon (\prod_{i\in[n]}X_i,G) \to (\prod_{i\in[n]}Y_i,G)$$ 
	by 
	$$\prod_{i\in[n]}\pi_i(x_1,x_2,\cdots,x_n)=(\pi_1 x_1,\pi_2 x_2,\cdots,\pi_n x_n)$$ 
	for every $(x_1,x_2,\cdots,x_n)\in \prod_{i\in[n]}X_i$. When $\pi_i=\pi$ for all $i\in[n]$, we write 
	$\prod_{i\in[n]}\pi_i$ as $\pi^{(n)}$.
	In \cite{Huang-Ye-Zhang-1}, Huang, Ye and Zhang showed that the finite product of rel-u.p.e factor 
	maps between $\mathbb{Z}$-systems has rel-u.p.e. It also holds for $G$-systems. 
	
	\begin{thm}\label{product is upe}Let $\pi_i\colon (X_i,G)\to (Y_i,G)$ be a factor map between two $G$-systems 
		and $supp(Y_i)=Y_i$ for $i=1,2$. Then $\pi_1$ and $\pi_2$ have rel-u.p.e if and only if 
		$\pi_1\times \pi_2\colon (X_1\times X_2,G)\to (Y_1\times Y_2,G)$ has rel-u.p.e.
	\end{thm}
	We will give a proof of Theorem \ref{product is upe} in Appendix \ref{B} (see Theorem \ref{Bpthm}).
	\section{$\pi$ has rel-u.p.e if and only if $\widetilde{\pi}_n$ has rel-u.p.e}\label{section 3}
	Let $X$ be a compact metric space and $\rho_X$ be a compatible metric for $X$. We denote
	$B_{\rho_X}(x,\delta)=\{y\in X\colon \rho_X(x,y)<\delta\}$ for $x\in X$ and $\delta>0$, and denote 
	$$\Delta(X)=\{(x,x)\colon x\in X\}.$$
    For $(x_1,x_2)\in X\times X\backslash\Delta(X)$ and $\mathcal{U}=\{U_1,U_2\}\in\mathcal{C}_X$, we say $\mathcal{U}$ is an \textbf{admissible cover of $X$ with respect to 
		$(x_1,x_2)$}, if for any $i\in[2]$, one has $\{x_1,x_2\}\nsubseteq\overline{U_i}$. Let $\pi:(X,G)\to (Y,G)$ be a factor 
	map between two $G$-systems. $(x_1,x_2)\in 
	X\times X\backslash\Delta(X)$ is called a \textbf{entropy pair relevant to $\pi$}, if for any admissible cover 
	$\mathcal{U}$ with respect to $(x_1,x_2)$ we have $h_{top}(\mathcal{U},G|\pi)>0$. Denote by $E(X,G|\pi)$ the set of 
	all entropy pairs relevant to $\pi$. Let 
	$$R_{\pi}=\{(x_1,x_2)\in X\times X \colon\pi(x_1)=\pi(x_2)\}.$$
	It is easy to
	see that $E(X,G|\pi)\subseteq R_{\pi}\setminus\Delta(X)$, and $\pi$ has rel-u.p.e iff $E(X,G|\pi)=R_{\pi}\setminus
	\Delta(X)$.

	Concept of dynamical independence is introduced in \cite[Definition 2.1]{Kerr-Li-1}. Now we consider its relative version.
	Let $\pi\colon (X,G)\to (Y,G)$ be a factor map between two $G$-systems. For any $n\in\mathbb{N}$ and a tuple
	$\mathcal{V}=(V_1,V_2,\cdots,V_n)$ of subsets of $X$, we say $J\subseteq G$ is an 
	\textbf{independence set of $\mathcal{V}$ with respect to $\pi$}, if for every non-empty finite subset 
	$I\subset J$ there exists $y\in Y$ such that 
	$$\pi^{-1}(y)\cap \bigcap_{g\in I}g^{-1}V_{\sigma(g)}\neq\emptyset$$ 
	holds for every $\sigma\in[n]^{I}$. We denote by $\mathcal{P}_{\mathcal{V}}^{\pi}$ the set of all 
	independence sets of $\mathcal{V}$ with respect to $\pi$. 
	
	\begin{rem}\label{rem1}
	For every $n\in\mathbb{N}$ and a tuple
	$\mathcal{V}=(V_1,V_2,\cdots,V_n)$ of subsets of $X$, if we set 
	$$\mathcal{I}_{\mathcal{V}}\colon \mathcal{F}(G)\to \mathbb{R}; \ \ 
	\mathcal{I}_{\mathcal{V}}(F)\colon=\max_{I\subseteq F,I\in \mathcal{P}_{\mathcal{V}}^{\pi}}\vert I\vert,$$
	then by Theorem \ref{amenable converges}, $\mathcal{I}_{\mathcal{V}}(F)/\vert F\vert$ convergences as $F$ becomes more and more invariant and this limit is equal
	to $\inf_{F}\frac{\mathcal{I}_{\mathcal{V}}(F)}{\vert F\vert}$, where $F$ ranges over 
	$\mathcal{F}(G)$. When this limit is positive, we say $\mathcal{V}$ is \textbf{independent with respect to $\pi$}. 
	\end{rem}
	The next lemma follows Lemma 3.4 in \cite{Kerr-Li-1} (see also
	\cite[Theorem 7.4]{Huang-Ye}).
	\begin{lem}\label{independence} 
		Let $\pi: (X,G)\to (Y,G)$ be a factor map between two $G$-systems, and $V_1, V_2$ be two disjoint subsets of $X$.
		If we set $\mathcal{U}=\{X\setminus V_1, X\setminus V_2\}$, then $h_{top}(\mathcal{U},G|\pi)>0$ if and
		only if $\{V_1,V_2\}$ is independent with respect to $\pi$.
	\end{lem}
	
	Let $\pi: (X,G)\to (Y,G)$ be a factor map between two $G$-systems.
	For any $(x_1,x_2)\in X\times X\backslash\Delta(X)$, disjoint open subsets $V_1, V_2$ of $X$ with $x_i\in V_i$ for $i\in[2]$,
    $\mathcal{V}=\{X\setminus V_1, X\setminus V_2\}$ is an admissible cover 
	of $X$ with respect to $(x_1,x_2)$. Then by Lemma \ref{independence}, we immediately have the following corollary.
	\begin{cor}\label{rel-entropy-pair}
		Let $\pi: (X,G)\to (Y,G)$ be a factor map between two $G$-systems and $(x_1,x_2)\in X\times X\backslash\Delta(X)$. 
		Then $(x_1,x_2)\in E(X,G|\pi)$ if and only if for any disjoint open subsets $V_1, V_2$ of $X$ with $x_i\in V_i$ for $i=1,2$, 
		$\{V_1,V_2\}$ is independent with respect to $\pi$.
	\end{cor}

	We note that, for any two nonempty finite sets $H$, $W$, if $H\subseteq W$ and $S\subset\{1,2\}^{W}$, one has
	\begin{align}\label{explain}
		\big\vert S\vert_H\big\vert\geq\frac{\vert S\vert}{2^{\vert W\vert-\vert H\vert}},
	\end{align}
	where $S\vert_H$ is the restriction of $S$ on $H$, i.e.
	$$S\vert_H=\bigg\{\sigma\in\{1,2\}^H\colon\text{there\ exists}\ \sigma'\in S\ \text{such\ that}\ 
	\sigma(h)=\sigma'(h)\ \text{for\ all}\ h\in H\bigg\}.$$
	The following consequence of Karpovsky and Milman's generalization of the 
	Sauer-Perles-Shelah lemma \cite{Karpovsky-Milman,Sauer,Shelah} is well known, 
	one can also see \cite[Lemma 3.5]{Kerr-Li-1}.
	
	\begin{lem}\label{F_S}
		Given $k\geq 2$ and $\lambda>1$ there exists a constant $c>0$ such that, for all $n\in\mathbb{N}$, if 
		$S\subseteq[k]^{[n]}$ satisfies $\vert S\vert\geq ((k-1)\lambda)^n$ then there is an 
		$I\subseteq[n]$ with $\vert I\vert\geq cn$ and $S\vert_I=[k]^I$.
	\end{lem}

	Theorem \ref{main} follows from Theorem \ref{pi-pi_n}, Theorem
	\ref{mainthm1} and Theorem \ref{mainthm2}.

	\begin{thm}\label{pi-pi_n} Let $n\in\mathbb{N}$, $\pi\colon (X,G)\to (Y,G)$ be a factor map between two $G$-systems,	
	$\widetilde{\pi}\colon (\mathcal{M}(X),G)\to (\mathcal{M}(Y),G)$ be the factor map induced by $\pi$ and                                                                                                                                            
	$\widetilde{\pi}_n\colon$ $(\mathcal{M}_n(X),G)$ $\to(\mathcal{M}_n(Y),G)$ be the restriction of $\widetilde{\pi}$
	on $\mathcal{M}_n(X)$. When $supp(Y)$ $=Y$, the following are equivalent
	\begin{itemize}
		\item[(1)] $\pi$ has rel-u.p.e; 
		\item[(2)]  $\widetilde{\pi}_n$ has rel-u.p.e for some $n\in\mathbb{N}$;
		\item[(3)] $\widetilde{\pi}_n$ has rel-u.p.e for every $n\in\mathbb{N}$.
   \end{itemize}
	\end{thm}
	\begin{proof}$(3)\Rightarrow (2)$ is trivial. We will prove $(1)\Rightarrow (3)$ and $(2)\Rightarrow (1)$.
	
		$(1)\Rightarrow (3)$ Assume that $\pi$ has rel-u.p.e. For every fixed $1\le n< \infty$, in order to obtain that $\widetilde{\pi}_n$ has rel-u.p.e, it is sufficient to prove that $E(\mathcal{M}_n(X),G|\widetilde{\pi}_n)\supseteq 
		R_{\widetilde{\pi}_n}(\mathcal{M}_{n}(X),G)$ $ \backslash\Delta(\mathcal{M}_n(X))$. 
		
		Let $(\mu_1,\mu_2)\in R_{\widetilde{\pi}_n}(\mathcal{M}_n(X),$ $G)\backslash\Delta(\mathcal{M}_n(X))$, and $\widetilde{V}_1$ and 
		$\widetilde{V}_2$ be two disjoint open subsets of $\mathcal{M}_n(X)$ with $\mu_i\in\widetilde{V}_i$ for $i\in[2]$.
        By Corollary \ref{rel-entropy-pair}, we shall show that $\{\widetilde{V}_1,\widetilde{V}_2\}$ is independent with respect to
		 $\widetilde{\pi}_n$.
		
		For $i\in[2]$ and $j\in[n]$, there exist points $x_j^i\in X$ such that $\mu_i=\frac{1}{n}\sum_{j=1}^n\delta_{x_j^i}$.
		We note that the map $\Phi\colon X^{(n)}\to\mathcal{M}(X)$, defined by 
		$$\Phi(z_1,z_2,\cdots,z_n)=\frac{1}{n}\sum_{i=1}^{n}\delta_{z_i}$$
		is continuous. Thus for every $i\in[2]$ and $j\in[n]$ there exists open neighbourhoods $V_j^i$ of $x_j^i$ 
		such that
		$$\mu_i\in\bigg\{\frac{1}{n}\sum_{j=1}^{n}\delta_{z_j}\colon z_j\in V_j^i, j\in[n]\bigg\}\subseteq \widetilde{V}_i.$$
		Since $\widetilde{V}_1\cap\widetilde{V}_2=\emptyset$, if we set $W_i=V^i_1\times V_2^i\times\cdots\times V_n^i$ for $i=1,2$, one has 
		$W_1\cap W_2=\emptyset$.
		Without loss of generality, we can assume that $\pi(x_j^1)=\pi(x_j^2)$ for all $j\in[n]$ since $\widetilde{\pi}_n(\mu_1)=\widetilde{\pi}_n(\mu_2)$.
        Let $\omega_i=(x_1^i,x_2^i,\cdots,x_n^i)\in W_i$ for $i=1,2$.
		Then 
		$$(\omega_1,\omega_2)\in R_{\pi^{(n)}}\setminus\Delta(X^{(n)})=E(X^{(n)},G\vert\pi^{(n)})$$ 
		as $\pi^{(n)}$ has rel-u.p.e by Theorem \ref{product is upe}. Thus $\{W_1,W_2\}$ is independent with respect with $\pi^{(n)}$. We note that
		$\mathcal{P}^{\pi^{(n)}}_{\{W_1,W_2\}}\subseteq\mathcal{P}^{\widetilde{\pi}_n}_{\{\widetilde{V_1},\widetilde{V_2}\}}$. 
		This implies $\{\widetilde{V_1},\widetilde{V_2}\}$ is independent with respect to $\widetilde{\pi}_n$.

		$(2)\Rightarrow (1)$ We assume that $\widetilde{\pi}_n$ has rel-u.p.e for some positive integer 
		$1\le n< \infty$. In the following, we prove that $R_{\pi}\setminus\Delta(X)\subseteq E(X,G\vert\pi)$.
		Let $(x_1,x_2)\in R_{\pi}\setminus\Delta(X)$, $V_1$ and $V_2$ be two disjoint open subsets of $X$ with $x_i\in V_i$, $i=1,2$.
		By Corollary \ref{rel-entropy-pair}, we only need to show that $\{V_1,V_2\}$ is independent with respect to $\pi$.
		
		We set 
		$$\widetilde{V_i}=\bigg\{\mu\in\mathcal{M}_n(X): \mu(V_i)>1-\frac{1}{2n}\bigg\}$$
		for $i=1,2$. Clearly, $\widetilde{V_1}$ and $\widetilde{V_2}$ are disjoint open subsets of $\mathcal{M}_n(X)$ with 
		$\delta_{x_i}\in\widetilde{V_i}$ for $i=1,2$. Since $\widetilde{\pi}_n$ has rel-u.p.e, and 
		$(\delta_{x_1},\delta_{x_2})\in R_{\widetilde{\pi}_n}\setminus\Delta(\mathcal{M}_n(X))=E(\mathcal{M}_n(X),G\vert \widetilde{\pi}_n)$,
		$\{\widetilde{V_1},\widetilde{V_2}\}$ is independent with respect to $\widetilde{\pi}_n$. 
		Then there exists constant $c>0$, such that
		for every fixed $F\in\mathcal{F}(G)$, there exist $I\subseteq F$ with $\vert I\vert>c\vert F\vert$ and $\nu=\frac{1}{n}\sum_{i=1}^{n}\delta_{y_i}\in\mathcal{M}_n(Y)$
		for some $y_i\in Y$ such that
		$$A_{\sigma}\colon=\widetilde{\pi}_n^{-1}(\nu)\cap \bigcap_{g\in I}g^{-1}\widetilde{V}_{\sigma(g)}\neq\emptyset,$$
		for every $\sigma\in\{1,2\}^{I}$.
		
		For every $\sigma\in\{1,2\}^{I}$, and $\mu_{\sigma}=\frac{1}{n}\sum_{i=1}^{n}\delta_{z_i^{\sigma}}\in A_{\sigma}$, we can assume $\pi(z_i^{\sigma})=y_i$ for $i\in[n]$.
		Moreover, for every $g\in I$ one has $g\mu_{\sigma}=\frac{1}{n}\sum_{i=1}^n\delta_{gz_i^{\sigma}}\in\widetilde{V}_{\sigma(g)}$. That is 
		$$\frac{1}{n}\sum_{i=1}^n\delta_{gz_i^{\sigma}}(V_{\sigma(g)})>1-\frac{1}{2n},$$
		which implies $gz_{i}^{\sigma}\in V_{\sigma(g)}$ for every $i\in[n]$. Particularly, 
		$$z_{1}^{\sigma}\in \pi^{-1}(y_1)\cap\bigcap_{g\in I}g^{-1}V_{\sigma(g)},$$
		for every $\sigma\in\{1,2\}^I$. Thus $\{V_1,V_2\}$ is independent with respect to $\pi$. This ends our proof.	
	\end{proof}
	
\section{$\pi$ is rel-u.p.e if and only if $\widetilde{\pi}$ is rel-u.p.e}\label{section 4}
In this section, we will prove $\pi$ is rel-u.p.e if and only if $\widetilde{\pi}$ is rel-u.p.e.
We need the following lemma.

\begin{lem}\label{R_pi-dense}Let $\pi\colon X\to Y$ be a continuous surjective map between two compact metric spaces,
	$\widetilde{\pi}\colon \mathcal{M}(X)\to \mathcal{M}(Y)$ be the map induced by $\pi$ and                                                                                                                                            
	$\widetilde{\pi}_n\colon\mathcal{M}_n(X)\to\mathcal{M}_n(Y)$ be the restriction of $\widetilde{\pi}$
	on $\mathcal{M}_n(X)$.
	Then $\bigcup_{n\in\mathbb{N}}R_{\widetilde{\pi}_n}$ is dense in $R_{\widetilde{\pi}}$. 
\end{lem}
\begin{proof}Fix compatible metrics $\rho_X$ for $X$ and $\rho_Y$ for $Y$ respectively. Let $(\mu_1,\mu_2)\in R_{\widetilde{\pi}}$. Without loss of generality, we can assume 
	$\mu_1\neq\mu_2$. For any two disjoint open subsets 
	$\widetilde{V}_1$, $\widetilde{V}_2$ of $\mathcal{M}(X)$ with $\mu_i\in\widetilde{V}_i$ for $i\in[2]$, by 
 (\ref{basis of M(X)-function}) there exist a  constant $r>0$ small enough, integers $L_1$ and $L_2$, 
$f_1,\cdots,f_{L_1}\in C(X)$ and $g_1,\cdots,g_{L_2}\in C(X)$ such that
$$\mu_1\in\widetilde{W}_1:=\bigg\{\mu\in\mathcal{M}(X):\vert\int_X f_id\mu-\int_X f_id\mu_1\vert<r,i\in [L_1]\bigg\}\subseteq\widetilde{V_1},$$
and
$$\mu_2\in\widetilde{W}_2:=\bigg\{\mu\in\mathcal{M}(X):\vert\int_X g_jd\mu-\int_X g_jd\mu_2\vert<r,j\in[L_2]\bigg\}\subseteq\widetilde{V_2}.$$
It is sufficient to prove that $\big(\widetilde{W}_1\times \widetilde{W}_2\big)\cap R_{\widetilde{\pi}_N}\neq\emptyset$ for
some $N\in\mathbb{N}$.

Without loss of generality, we can assume $\Vert f_i\Vert\leq 1$ and $\Vert g_j\Vert\leq 1$ for $i\in[L_1]$ and $j\in[L_2]$.
Moreover, since $f_i,g_j\in C(X)$ for $i\in[L_1]$ and $j\in[L_2]$, there exists $\varepsilon>0$ such that for any 
$x,z\in X$ with $\rho_X(x,z)<\varepsilon$, one has 
\begin{equation}\label{Uniform continuous}
	\begin{array}{l}\vert f_i(x)-f_i(z)\vert<\frac{r}{2},\ \text{for\ every }\ \ i\in [L_1], \\
	\vert g_j(x)-g_j(z)\vert<\frac{r}{2},\ \text{for\ every }\ \ j\in[L_2].
\end{array}\end{equation}

For every $y\in Y$, since  $\pi$ is continuous, one can find an open neighbourhood 
$V_y\subseteq Y$ such that 
$$\pi^{-1}(y)\subseteq\pi^{-1}(V_y)\subseteq\overline{\pi^{-1}(V_y)}\subseteq(\pi^{-1}(y))^{\frac{\varepsilon}{2}},$$ 
where $(\pi^{-1}(y))^{\frac{\varepsilon}{2}}=\{x\in X: \rho_X(x,\{\pi^{-1}(y)\})<\frac{\varepsilon}{2}\}$. Moreover, since $Y$ is compact, there exist $K\in\mathbb{N}$ and pairwise different points 
$y_1,\cdots,y_K$ of $Y$ such that $Y=\cup_{i=1}^{K}V_{y_i}$. Then one can find $t>0$ such that 
$y_i\in B_{\rho_Y}(y_i,t)\subset V_{y_i}$ for any $i\in[K]$ and $\{B_{\rho_Y}(y_1,t),\cdots,B_{\rho_Y}(y_K,t)\}$ are pairwise disjoint.
We set 
$$W_1= V_{y_1}\setminus\bigcup_{i=2}^{K}B_{\rho_Y}(y_i,t) \ \text{and}\  
W_i= V_{y_i}\setminus\bigg(\bigcup_{j=1}^{i-1}V_{y_j}\cup\bigcup_{j=i+1}^{K}B_{\rho_Y}(y_j,t)\bigg)
$$ for $i=2,\cdots,K$. Then $\{W_1,\cdots,W_K\}$ is a partition of $Y$ and $y_i\in W_i\subseteq  V_{y_i}$ for $i\in[K]$.
Moreover, 
$\{\pi^{-1}(W_1),\cdots,\pi^{-1}(W_K)\}$ is a partition of $X$ which satisfies 
$$ \pi^{-1}(y_i)\subseteq\pi^{-1}(W_i)\subseteq \overline{\pi^{-1}(V_{y_i})}\subseteq(\pi^{-1}(y_i))^{\frac{\varepsilon}{2}}$$
for every $i\in[K]$. Then for every $i\in[K]$, there eixst $P_i\in\mathbb{N}$ and pairwise different $x_1^i,x_2^i,\cdots,x_{P_i}^i\in\pi^{-1}(y_i)$, such that 
$\{x_j^i\colon j\in[P_i]\}$ is a $\frac{\varepsilon}{2}$-net of $\pi^{-1}(W_i)$. Then one can choose Borel subsets $A_j^i$ 
of $X$ for $i\in[K]$ and $j\in[P_i]$, such that
\begin{enumerate}
	\item[(i)] $diam(A_j^i)<\varepsilon$ for every $i\in[K],j\in[P_i]$;
	\item[(ii)] $x_j^i\in A_j^i$ for every $i\in[K], j\in[P_i]$;
	\item[(iii)] $\{A_j^i\colon j\in[P_i]\} $ is a partition of $\pi^{-1}(W_i)$ for every $i\in[K]$.
\end{enumerate}

For every $i\in[K]$, $j\in[P_i]$, we set $a_{ij}=\mu_{1}(A_j^{i})$ and $b_{ij}=\mu_{2}(A_j^{i})$.
Since $\widetilde{\pi}(\mu_1)=\widetilde{\pi}(\mu_2)$, we have 
$$\sum_{j=1}^{P_i}a_{ij}=\sum_{j=1}^{P_i}\mu_{1}(A_j^{i})=\mu_{1}(\pi^{-1}(W_i))=\mu_{2}(\pi^{-1}(W_i))=\sum_{j=1}^{P_i}\mu_{2}(A_j^{i})=\sum_{j=1}^{P_i}b_{ij}$$
for $i\in[K]$.
Then for any $i\in[K]$ and $j\in[P_i]$, there exist integers $q_{ij}$, $\tilde{q}_{ij}$, $Q_i$ and  
$N\in\mathbb{N}$ large enough satisfying the following conditions
\begin{enumerate}
\item[(i*)] $\frac{q_{ij}}{N}\leq a_{ij}<\frac{q_{ij}+1}{N}$;
\item[(ii*)] $\frac{\tilde{q}_{ij}}{N}\leq b_{ij}<\frac{\tilde{q}_{ij}+1}{N}$;
\item[(iii*)] $\frac{Q_i}{N}\leq\sum_{j=1}^{P_i}a_{ij}=\sum_{j=1}^{P_i}b_{ij}<\frac{Q_i+1}{N}$.
\end{enumerate}
 Now, we choose a $x_0\in X$ arbitrarily and set	 
 \begin{footnotesize}   
 $$\widetilde{\mu}_1=\frac{1}{N}\bigg(\sum_{i=1}^{K}\bigg(\sum_{j=1}^{P_i-1}q_{ij}\delta_{x_{j}^{i}}+\big(Q_i-\sum_{j=1}^{P_i-1}q_{ij}\big)\delta_{x_{P_i}^{i}}\bigg)\bigg)+\frac{N-\sum_{i=1}^{K}Q_{i}}{N}\delta_{x_0},$$
 \end{footnotesize}
 and
 \begin{footnotesize}
$$\widetilde{\mu}_2=\frac{1}{N}\bigg(\sum_{i=1}^{K}\bigg(\sum_{j=1}^{P_i-1}\tilde{q}_{ij}\delta_{x_{j}^{i}}+\big(Q_i-\sum_{j=1}^{P_i-1}\tilde{q}_{ij}\big)\delta_{x_{P_i}^{i}}\bigg)\bigg)+\frac{N-\sum_{i=1}^{K}Q_{i}}{N}\delta_{x_0}.$$
 \end{footnotesize}
 It is clear that $(\widetilde{\mu}_1,\widetilde{\mu}_2)\in R_{\widetilde{\pi}_N}$. Now we shall show that 
$\widetilde{\mu}_i\in\widetilde{W_i}$ for $i\in[2]$. 

In fact, for any $\ell\in[L_1]$ one has
\begin{footnotesize}
\begin{align}\label{4-4-1}
\bigg\vert\int f_{\ell}d\mu_1-\int f_{\ell}d\tilde{\mu}_1\bigg\vert&=
\bigg\vert\sum_{i=1}^{K}\sum_{j=1}^{P_i}\int_{A_{j}^{i}}f_ld\mu_1-\frac{1}{N}\sum_{i=1}^{K}\bigg(\sum_{j=1}^{P_i-1}q_{ij}f_l(x_{j}^{i})\nonumber\\
&+\big(Q_i-\sum_{j=1}^{P_i-1}q_{ij}\big)f_l(x_{P_i}^{i})\bigg)-\frac{N-\sum_{i=1}^{K}Q_{i}}{N}f_l(x_0)\bigg\vert\nonumber\\
&\leq \bigg\vert\sum_{i=1}^{K}\sum_{j=1}^{P_i}\int_{A_{j}^{i}}f_ld\mu_1-\frac{1}{N}\sum_{i=1}^{K}\sum_{j=1}^{P_i}q_{ij}f_l(x_{j}^{i})\bigg\vert\\
&+\bigg\vert\frac{1}{N}\sum_{i=1}^{K}(Q_i-\sum_{j=1}^{P_i}q_{ij})f_l(x_{P_i}^{i})\bigg\vert+\bigg\vert\frac{N-\sum_{i=1}^{K}Q_{i}}{N}f_l(x_0)\bigg\vert\nonumber.
\end{align}
\end{footnotesize}
Since $diam(A_j^{i})<\varepsilon$ for $i\in[K]$ and $j\in[P_i]$, by (\ref{Uniform continuous}) and 
$(\text{i}^*)$, we have
\begin{small}
\begin{align}\label{4-4-2}
&\bigg\vert\sum_{i=1}^{K}\sum_{j=1}^{P_i}\int_{A_{j}^{i}}f_l(x)d\mu_1-\frac{1}{N}\sum_{i=1}^{K}\sum_{j=1}^{P_i}q_{ij}f_l(x_{j}^{i})\bigg\vert\nonumber\\
&\leq\sum_{i=1}^{K}\sum_{j=1}^{P_i}\int_{A_{j}^{i}}\big\vert f_l(x)-f_l(x_{j}^{i})\big\vert d\mu_1+\sum_{i=1}^{K}\sum_{j=1}^{P_i}(a_{ij}-\frac{q_{ij}}{N})\vert f_l(x_{j}^{i})\vert\nonumber\\
&\leq\frac{r}{2}+\frac{\sum_{i=1}^{K}P_i}{N}.
\end{align}
\end{small}
By $(\text{i}^*)$ and $(\text{iii}^*)$ one has
\begin{footnotesize}
\begin{align}\label{4-4-3}
 \sum_{i=1}^{K}\bigg\vert\frac{Q_{i}}{N}-\frac{1}{N}\sum_{j=1}^{P_i}q_{ij}\bigg\vert
&\leq\sum_{i=1}^{K}\bigg\vert\frac{Q_{i}}{N}-\sum_{j=1}^{P_i}a_{ij}\bigg\vert+\sum_{i=1}^{K}\bigg\vert\sum_{j=1}^{P_i}a_{ij}-\frac{1}{N}\sum_{j=1}^{P_i}q_{ij}\bigg\vert\\
&\leq\frac{K}{N}+\frac{\sum_{i=1}^{K}P_i}{N}\nonumber,
\end{align}
\end{footnotesize}
and
\begin{footnotesize}
\begin{align}\label{4-4-4}
	\bigg\vert\frac{N-\sum_{i=1}^{K}Q_{i}}{N}\bigg\vert
 &=\bigg\vert\sum_{i=1}^{K}\sum_{j=1}^{P_i}a_{ij}-\frac{1}{N}\sum_{i=1}^{K}Q_{i}\bigg\vert\nonumber\leq\sum_{i=1}^{K}\bigg\vert\sum_{j=1}^{P_i}a_{ij}-\frac{Q_{i}}{N}\bigg\vert\leq\frac{K}{N}.
\end{align}
\end{footnotesize}
When $N$ is large enough such that 
$\frac{K}{N}+\frac{\sum_{i=1}^{K}P_i}{N}\leq\frac{r}{6},$
by (\ref{4-4-1}), (\ref{4-4-2}) and (\ref{4-4-3}), we have $\widetilde{\mu}_1\in\widetilde{W_1}$. Similarly, we can prove that $\widetilde{\mu}_2\in\widetilde{W_2}$. This ends
our proof.
\end{proof}

\begin{thm}\label{mainthm1}Let $\pi:(X,G)\to(Y,G)$ be a factor map between two $G$-systems with $supp(Y)=Y$ and 
$\widetilde{\pi}:(\mathcal{M}(X),G)\to(\mathcal{M}(Y),G)$ be the induced map of $\pi$. 
Suppose $\pi$ has rel-u.p.e then $\widetilde{\pi}$ also has rel-u.p.e.
\end{thm}
\begin{proof}
Assume that $\pi$ has rel-u.p.e. To show $\widetilde{\pi}$ has rel-u.p.e, it suffices to prove that 
$R_{\widetilde{\pi}}\setminus\Delta(\mathcal{M}(X))\subseteq E(\mathcal{M}(X),G\vert\widetilde{\pi})$.
Let $(\mu_1,\mu_2)\in R_{\widetilde{\pi}}\setminus\Delta(\mathcal{M}(X))$, and
$\widetilde{V_1}$, $\widetilde{V_2}$ be two disjoint open subsets of $\mathcal{M}(X)$ with $\mu_i\in\widetilde{V}_i$ for $i\in[2]$. 
By Lemma \ref{R_pi-dense}, there exist $n\in \mathbb{N}$ and 
$(\mu_1',\mu_2')\in R_{\widetilde{\pi}_n}\cap \big(\widetilde{V_1}\times \widetilde{V_2}\big)$.
Notice that, since $\pi$ has rel-u.p.e, by Theorem \ref{pi-pi_n}, $\widetilde{\pi}_n$ has rel-u.p.e. 
Then 
$\{\widetilde{V_1}\cap\mathcal{M}_n(X),\widetilde{V_2}\cap\mathcal{M}_n(X)\}$ is independent with respect to 
$\widetilde{\pi}_n$, which implies $\{\widetilde{V_1},\widetilde{V_2}\}$ is independent with respect to $\widetilde{\pi}$. This ends our proof. 
\end{proof}
	
We note that for any non-empty finite subsets $A$, $H$ of $\mathbb{N}$ with $A\subseteq H$ and 
$S\subseteq \{1,2\}^H$, one can find $S_0\subset S$ with $\vert S_0\vert\geq \frac{\vert S\vert}{
		2^{\vert H\vert-\vert A\vert}}$ such that for every $\sigma_1\neq\sigma_2\in S_0$, there exists $a\in A$ with
	\begin{equation}\label{e8}
		\sigma_1(a)\neq\sigma_2(a).
	\end{equation}
	In fact, if we let $\mathcal{W}=S\vert_A$, then $\vert \mathcal{W}\vert\geq \frac{\vert S\vert}{2^{\vert H\vert-\vert A\vert}}$. For each $w\in\mathcal{W}$, 
	there exists $\sigma_w\in S$ such that $\sigma_w\vert_A=w$. Put $S_0\colon=\{\sigma_w\colon w\in \mathcal{W}\}\subseteq S$.
	Then $\vert S_0\vert=\vert\mathcal{W}\vert\geq \frac{\vert S\vert}{2^{\vert H\vert-\vert A\vert}}$, and for every
	$\sigma_1\neq\sigma_2\in S_0$ one has $\sigma_1\vert_A\neq\sigma_2\vert_A$.
	
	\begin{thm}\label{mainthm2}Let $\pi:(X,G)\to(Y,G)$ be a factor map between two $G$-systems and 
		$\widetilde{\pi}:(\mathcal{M}(X),G)\to(\mathcal{M}(Y),G)$ be the induced map of $\pi$. 
		If $\widetilde{\pi}$ has rel-u.p.e, then so does $\pi$.
	\end{thm}
	\begin{proof} 
		Assume that $\widetilde{\pi}$ has rel-u.p.e. To show $\pi$ has rel-u.p.e, we shall show that   
		$R_{\pi}\setminus\Delta(X)\subseteq E(X,G\vert\pi)$. 
		Let $(x_1,x_2)\in R_{\pi}\setminus\Delta(X)$, $V_1,V_2$ be two nonempty disjoint open subsets of $X$ with
		$x_i\in V_i$ for $i\in[2]$. By Corollary \ref{rel-entropy-pair}, it is sufficient to show that
		$(V_1,V_2)$ is independent with respect to $\pi$.

		Take $\epsilon\in(0,\frac{1}{2})$ with 
		\begin{align}\label{4.5-0}
			2^{1-\epsilon^2}\cdot (1-\epsilon^2)^{(1-\epsilon^2)}\cdot(\epsilon^2)^{(\epsilon^2)}>1.
		\end{align}
		We set
			\begin{align}\label{V}
			&\widetilde{V_i}=\{\mu\in\mathcal{M}(X): \mu(V_i)> 1-\epsilon^{4}\}
		\end{align}
		for $i\in[2]$. Clearly, $\delta_{x_i}\in\widetilde{V}_i$. Since
		$(\delta_{x_1},\delta_{x_2})\in R_{\widetilde{\pi}}$ and $\widetilde{\pi}$ has rel-u.p.e, $(\widetilde{V}_1,\widetilde{V}_2)$ is independent
		with respect to $\widetilde{\pi}$. That is, there exists $c>0$, such that for every $F\in\mathcal{F}(G)$
		there exists an independence set $E\subseteq F$ of $(\widetilde{V}_1,\widetilde{V}_2)$ with respect to 
		$\widetilde{\pi}$ with $\vert E\vert>c\vert F\vert$.

		Fix a $F\in\mathcal{F}(G)$ and an independence set $E\subseteq F$ of $(\widetilde{V}_1,\widetilde{V}_2)$ with respect to 
		$\widetilde{\pi}$ with $\vert E\vert>c\vert F\vert$. Then there exists $\nu\in\mathcal{M}(Y)$, such that 
		for every $\sigma\in\{1,2\}^E$
		\begin{align}\label{thm 4.5-1} 
			\boldsymbol{\widetilde{V}}_{\sigma}\colon=\bigg(\bigcap_{g\in E}g^{-1}\widetilde{V}_{\sigma(g)}\bigg)\cap\widetilde{\pi}^{-1}(\nu)\neq\emptyset.
		\end{align}
		For every $\sigma\in\{1,2\}^{E}$, we take $\mu_{\sigma}\in\boldsymbol{\widetilde{V}}_{\sigma}$.
		Then 		
		$\mu_{\sigma}\in g^{-1}\widetilde{V}_{\sigma(g)}$ for every $g\in E$ and $\sigma\in\{1,2\}^{E}$,
		which implies
		$\mu_{\sigma}(g^{-1}V_{\sigma})>1-\epsilon^4$ for every $g\in E$ and $\sigma\in\{1,2\}^E$.
		Thus 
	   $$\int_{X}\frac{1}{\vert E\vert}\sum_{g\in E}1_{g^{-1}V_{\sigma(g)}}(x)d{ \mu_{\sigma}}=\frac{1}{|E|}\sum_{g\in E}\mu_{\sigma}(g^{-1}V_{\sigma(g)})> 1-\epsilon^{4},$$
	   and $\mu_{\sigma}(\widetilde{X}_{\sigma})> 1-\epsilon^{2}$ for every $\sigma\in\{1,2\}^{E}$, where 
	   \begin{align}\label{thm 4.5-17}
		\widetilde{X}_{\sigma}=\{x\in X: \frac{1}{|E|}\sum_{g\in E}1_{g^{-1}V_{\sigma(g)}}(x)> 1-\epsilon^{2}\}.
	   \end{align}
       By the inner regular of measure, we can find a closed subset 
	   \begin{align}\label{thm 4.5-18}
		X_{\sigma}\subseteq\widetilde{X}_{\sigma} \quad \text{with} \quad
	   \mu_{\sigma}(X_{\sigma})> 1-\epsilon^{2}
	   \end{align}
	   for every $\sigma\in\{1,2\}^{E}$. 
	   Since $\pi$ is continuous, for every $\sigma\in\{1,2\}^E$ we have 
	   \begin{align}\label{thm 4.5-20}
		Y_{\sigma}\colon=\pi(X_{\sigma})
	   \end{align}
	    is a closed subset of $Y$ and
		$$\nu(Y_{\sigma})=\widetilde{\pi}\mu_{\sigma}(Y_{\sigma})\geq\mu_{\sigma}(X_{\sigma})> 1-\epsilon^{2}.$$ 
		Then
		$$\int_{Y}\frac{1}{2^{|E|}}\sum_{\sigma\in \{1,2\}^{E}}1_{Y_{\sigma}}(y)d \nu> 1-\epsilon^{2}.$$ 
		Put 
		\begin{align}\label{thm 4.5-24}
			Y^{'}:=\{y\in Y: \frac{1}{2^{|E|}}\sum_{\sigma\in \{1,2\}^{E}}1_{Y_{\sigma}}(y)> 1-\epsilon\},
		\end{align}
		then $\nu(Y^{'})> 1-\epsilon>\frac{1}{2}$.

        Now, we fix a point $y_0\in  Y^{'}$ and set 
		\begin{align}\label{thm 4.5-26}
			\mathcal{E}\colon=\bigg\{\sigma\in\{1,2\}^{E}: y_0\in Y_{\sigma}\bigg\}.
		\end{align}
		Then $|\mathcal{E}|>(1-\epsilon)\cdot2^{\vert E\vert}$ by (\ref{thm 4.5-24}). 
		For any $\sigma\in\mathcal{E}$, by (\ref{thm 4.5-26}),(\ref{thm 4.5-20}),(\ref{thm 4.5-18}) and (\ref{thm 4.5-17}) there is $x_{\sigma}\in X_{\sigma}$ with 
		$$\frac{1}{|E|}\sum_{g\in E}1_{g^{-1}V_{\sigma(g)}}(x_{\sigma})>1-\epsilon^2$$ 
		such that $\pi(x_{\sigma})=y_0$. For every $\sigma\in\mathcal{E}$ we set
		$$A(\sigma)=\left\{g\in E: x_{\sigma}\in g^{-1}V_{\sigma(g)}\right\},$$
		then $\vert A(\sigma)\vert>(1-\epsilon^2)\vert E\vert$. Now we define
		$$\Omega\colon=\bigg\{H\subseteq E\colon\vert H\vert=\big\lfloor 
		(1-\epsilon^2)\cdot\vert E\vert\big\rfloor\bigg\},$$
		and 
		$$\mathcal{Q} (H)=\bigg\{\sigma\in\mathcal{E}:H\subseteq\{g\in E:gx_{\sigma}\in V_{\sigma(g)}\}\bigg\}$$
		for every $H\in\Omega$. Then $\vert\Omega\vert=\tbinom{\vert 
			E\vert}{\lfloor (1-\epsilon^2)\cdot\vert E\vert\rfloor}$ and $\bigcup_{H\in\Omega}\mathcal{Q} (H)=\mathcal{E}$.
	    Thus there exists $H_0\in\Omega$ such that $|\mathcal{Q} (H_0)|\geq\frac{|\mathcal{E}|}{|\Omega|}\geq\frac{(1-\epsilon)2^{|E|}}{\tbinom{\vert 
				E\vert}{\lfloor (1-\epsilon^2)\cdot\vert E\vert\rfloor}}.$ 
		By (\ref{e8}), we can choose $S\subseteq\mathcal{Q}(H_0)$ such that
			\begin{align}\label{4.5-5} 
			|S|\geq\frac{(1-\epsilon)2^{|E|}}{2^{|E|-\lfloor (1-\epsilon^2)\cdot\vert E\vert\rfloor} \cdot\tbinom{|E|}{\lfloor (1-\epsilon^2)\cdot\vert E\vert\rfloor}
			}
			\end{align}
			and for any $\sigma'\neq\sigma''\in S$, there exists $g\in H_0$ satisfies $\sigma'(g)\neq\sigma''(g)$. That is $\vert S\vert_{H_0}\vert=\vert S\vert$.
			Let $t=1-\epsilon^2$ and $\lambda=\log_2\big(2^t\cdot t^t\cdot (1-t)^{(1-t)}\big)>0$. Then by Stirling’s formula, when $\vert E\vert$ is large enough one has
			\begin{align*}
				&\vert S\vert_{H_0}\vert
				\approx 2^{\lfloor t\vert E\vert\rfloor-1}\cdot \sqrt{2\pi t(1-t)\vert E\vert}\cdot t^{t\vert E\vert}\cdot(1-t)^{(1-t)\vert E\vert}\\
				&\geq 2^{t\vert E\vert-2}\cdot \sqrt{2\pi t(1-t)\vert E\vert}\cdot t^{t\vert E\vert}\cdot (1-t)^{(1-t)\vert E\vert}\\
				&\geq \bigg(2^t\cdot t^t\cdot(1-t)^{(1-t)}\bigg)^{\vert E\vert}
				>2^{\lambda\vert E\vert}.
				\end{align*}
			By Lemma \ref{F_S}, there exists a subset $H_1\subseteq H_0$ with $\vert H_1\vert>d\vert H_0\vert$
			such that $S\vert_{H_1}=\{1,2\}^{H_1}$, where $d$ is a positive constant independent with $ E$ when 
			$\vert E\vert$ is lager enough. By Remark \ref{rem1}, $(\mathcal{V}_1,\mathcal{V}_2)$ is independent 
			with respect to $\pi$. This ends our proof.
		\end{proof}
\section{$\pi$ is open if and only if $\widetilde{\pi}$ is open}\label{section 5}
In this section, we will prove Theorem \ref{prop-1.4}. In fact, we have the following result.
\begin{thm}Let $\pi\colon X\to Y$ be a surjective continuous map between two compact metrizable spaces, 
	$\widetilde{\pi}\colon \mathcal{M}(X)\to\mathcal{M}(Y)$ be the induced map of $\pi$, and
	$\widetilde{\pi}_n\colon \mathcal{M}_n(X)\to\mathcal{M}_n(Y)$ be the restriction of $\widetilde{\pi}$
	on $\mathcal{M}_n(X)$. Then the following are equivalent
	\begin{enumerate}
		\item[(1)] $\pi$ is open.
		\item[(2)] $\widetilde{\pi}$ is open.
		\item[(3)] $\widetilde{\pi}_n$ is open for each $n\in\mathbb{N}$.
		\item[(4)] $\widetilde{\pi}_n$ is open for some $n\in\mathbb{N}$.
	\end{enumerate}
\end{thm}	
\begin{proof} $(3)\Rightarrow (4)$ is trivial. We will show $(2)\Rightarrow(1)$, $(4)\Rightarrow (1)$, $(1)\Rightarrow (3)$ and $(1)\Rightarrow(2)$.
Fix compatible metrics $\rho_X$ for $X$ and $\rho_Y$ for $Y$ respectively.

	(2)$\Rightarrow$ (1) Suppose that $\widetilde{\pi}$ is open. 
	For every nonempty open subset $U$ of $X$, we shall show that $\pi(U)$ is an open subset of $Y$. That is, for 
	every $y\in\pi(U)$ there exists $r>0$ such that $B_{\rho_Y}(y,r)\subseteq \pi(U)$.
	
	Now fix $y_0\in \pi(U)$. Since $U$ is open, there exist $x_0\in U$ and $\delta>0$ with $\pi(x_0)=y_0$ and 
	$\overline{B_{\rho_X}(x_0,\delta)}\subseteq U$. Then by Urysohn Lemma, there exists a countinuous map $f\colon X\to [0,1]$ 
	with $f(z)=1$ when $z\in B_{\rho_X}(x_0,\frac{\delta}{2})$ and $f(z)=0$ when $z\in X\backslash B_{\rho_X}(x_0,\delta)$.
	We set 
	$$\widetilde{U}\colon=\bigg\{\mu\in\mathcal{M}(X)\colon \int fd\mu>\frac{2}{3}\bigg\}.$$
	Clearly, $\widetilde{U}$ is an open subset of $\mathcal{M}(X)$ and $\delta_{x_0}\in\widetilde{U}$. 
	
	Since $\widetilde{\pi}$ is open, $\widetilde{\pi}(\widetilde{U})$ is an open subset of $\mathcal{M}(Y)$. Note that $\delta_{y_0}=\widetilde{\pi}(\delta_{x_0})\in\widetilde{\pi}(\widetilde{U})$.
	Thus there exists $r>0$ such that 
	$$\big\{\delta_{y}\colon \ y\in B_{\rho_Y}(y_0,r)\big\}\subset \widetilde{\pi}(\widetilde{U}).$$
	Then for every $y'\in B_{\rho_{Y}}(y_0,r)$, there exists $\mu_{y'}\in\widetilde{U}$ such that $\widetilde{\pi}(\mu_{y'})=\delta_{y'}$.
	On the one hand, since $\mu_{y'}(\{\pi^{-1}(y')\})=\delta_{y'}(\{y'\})=1$, we have 
	\begin{align}\label{thm4.1-1}
		supp(\mu_{y'})\subseteq\pi^{-1}(\{y'\}).
	\end{align}
	On the other hand, since $\mu_{y'}\in\widetilde{U}$, we have $\int fd\mu_{y'}>\frac{2}{3}$. Thus 
	$$\emptyset\neq supp(\mu_{y'})\cap B_{\rho_X}(x_0,\delta)\subseteq supp(\mu_{y'})\cap U.$$
	By (\ref{thm4.1-1}), we have $U\cap\pi^{-1}(\{y'\})\neq\emptyset$. That is $y'\in \pi(U)$. Then by the arbitrariness of $y'\in B_{\rho_{Y}}(y_0,r)$, one has $B_{\rho_{Y}}(y_0,r)\subseteq \pi(U)$. Thus 
	$\pi(U)$ is an open subset of $Y$ and $\pi$ is open.

	(4)$\Rightarrow $(1). We assume that there exists $n\in\mathbb{N}$ such that $\widetilde{\pi}_n$ is open.
		Let $U$ be an open subset of $X$. We shall show that for every $y\in\pi(U)$ there exists $r>0$ such that 
		$y\in B_{\rho_Y}(y,r)\subseteq\pi(U)$.
		
		Let $y\in\pi(U)$, there exists $x\in U$ with $\pi(x)=y$. We set 
		$$\widetilde{U}=\mathcal{M}_n(X)\cap\{\mu\in\mathcal{M}(X)\colon \mu(U)>0\}.$$
		$\widetilde{U}$ is an open subset of $\mathcal{M}_n(X)$ contains $\delta_x$. Since $\widetilde{\pi}_n$
		is open, $\widetilde{\pi}_n(\widetilde{U})$ is open which contains $\delta_y$. Then there exists $r>0$
		such that $\{\delta_{z}\colon\ \rho_{Y}(z,y)<r\}\subseteq \widetilde{\pi}_n(\widetilde{U})$. 
		Hence, for every $z\in B_{\rho_Y}(y,r)$, there exist $x_1,x_2,\cdots,x_n\in X$ such that
		$$\mu=\frac{1}{n}\sum_{i=1}^n\delta_{x_i}\in\widetilde{U}\ \ \text{and}\  \widetilde{\pi}_n(\mu)=\delta_z.$$ 
		Then one has $\pi(x_i)=z$ for every 
		$i\in[n]$. Since $\mu\in\widetilde{U}$, there exists $i_0\in[n]$ with 
		$x_{i_0}\in U$. That is, $z=\pi(x_{i_0})\in\pi(U)$. Hence, 
		$B_{\rho_Y}(y,r)\subseteq\pi(U)$. This implies $\pi$ is open.

		(1)$\Rightarrow $(3). Now we assume that $\pi$ is open. Let $n\in\mathbb{N}$ and 
		$\widetilde{U}$ be an open subset of $\mathcal{M}_n(X)$. We shall show that for every 
		$\nu\in\widetilde{\pi}_n(\widetilde{U})\subseteq \mathcal{M}_n(Y)$, there exists an open neighbourhood
		of $\nu$ in $\mathcal{M}_n(Y)$ contained in $\widetilde{\pi}_n(\widetilde{U})$.
		
		For any $\nu\in\widetilde{\pi}_n(\widetilde{U})\subseteq \mathcal{M}_n(Y)$, there exist positive integers $h$, $k_1,k_2,\cdots,k_h$ with
		$\sum_{i\in[h]}k_i=n$ and 
		and pairwise distinct $y_1,y_2,\cdots,y_h\in Y$ such that
		$$\nu=\frac{1}{n}(k_1\delta_{y_1}+k_2\delta_{y_2}+\cdots+k_h\delta_{y_h}).$$
		Since $\nu\in\widetilde{\pi}_n(\widetilde{U})$, there exists $\mu\in \widetilde{U}\subseteq\mathcal{M}_n(X)$ such that
		$\widetilde{\pi}_n(\mu)=\frac{1}{n}\sum_{i=1}^{h}k_i\delta_{y_i}$. Then for every $i\in[h]$ there exist integers $\ell_i$,
		$m_{i,j}$, and points $x_{i,j}\in X$ for $j\in[\ell_i]$ satisfying
		\begin{itemize}
			\item[(a)] $m_{i,1}+m_{i,2}+\cdots+m_{i,\ell_i}=k_i$ for every $i\in[h]$,
			\item[(b)] $x_{i,1},x_{i,2},\cdots,x_{i,\ell_i}$ are pairwise distinct and $\pi(x_{i,j})=y_i$ for every $i\in[h]$ and $j\in[\ell_i]$,
			\item[(c)] $\mu=\frac{1}{n}\sum\limits_{i\in[h]}\sum\limits_{j\in[\ell_i]}m_{i,j}\delta_{x_{i,j}}$.
		\end{itemize}
		Since $\widetilde{U}$ is an open neighbourhood of $\mu$, there exists $r_0>0$ such that if 
		$z_{i,j}^1,z_{i,j}^2,\cdots,$ $z_{i,j}^{m_{i,j}}\in B_{\rho_X}(x_{i,j},r_0)$
		for every $i\in[h]$, $j\in[\ell_i]$, then 
		\begin{align}\label{Ap-8}
			\frac{1}{n}\sum\limits_{i\in[h]}\sum\limits_{j\in[\ell_i]}\big(\sum_{t\in[m_{i,j}]}\delta_{z_{i,j}^t}\big)\in\widetilde{U}.
		\end{align}

		Note that, $y_1,y_2,\cdots,y_{h}$ are pairwise distinct, then there exists $\delta>0$ such that
		$\{B_{\rho_Y}(y_i,\delta)\}_{i\in [h]}$ are pairwise disjoint.
		By (b) and the continuity of $\pi$, there exists $r\in(0,r_0)$ such that
		\begin{align}\label{Ap-2}
			\pi(B_{\rho_X}(x_{i,j},r))\subseteq B_{\rho_X}(y_i,\delta)\ \ \  \text{and}\ \ \ B_{\rho_X}(x_{i,t},r)\cap B_{\rho_X}(x_{i,j},r)=\emptyset
		\end{align}
		for every $i\in[h]$, and different $j,t\in[\ell_i]$.
		
		Since $\pi$ is open,
		$\bigcap_{j=1}^{\ell_i}\pi\big(B_{\rho_X}(x_{i,j},r)\big)$ for every $i\in[h]$ is open. We set 
		$$\widetilde{V}\colon=\bigg\{\tau\in\mathcal{M}_n(Y)\colon\tau\bigg(\bigcap_{j=1}^{\ell_i}
		\pi\big(B_{\rho_X}(x_{i,j},r)\big)\bigg)>\frac{k_i}{n}-\frac{1}{2n},\ \ i\in[h]\bigg\},$$
        It is an open subset of $\mathcal{M}_n(Y)$. Moreover, for every $i_0\in[h]$,
		\begin{align*}
			\nu\bigg(\bigcap_{j=1}^{\ell_{i_0}}\pi\big(B_{\rho_X}(x_{i_0,j},r)\big)\bigg)
			&=\frac{1}{n}\sum_{i\in[h]}k_i\delta_{y_i}\bigg(\bigcap_{j=1}^{\ell_{i_0}}\pi\big(B_{\rho_X}(x_{i_0,j},r)\big)\bigg)\nonumber\\
			&=\frac{1}{n}k_{i_0}\delta_{y_{i_0}}\bigg(\bigcap_{j=1}^{\ell_{i_0}}\pi\big(B_{\rho_X}(x_{i_0,j},r)\big)\bigg)\nonumber\\
			&=\frac{1}{n}k_{i_0}>\frac{k_{i_0}}{n}-\frac{1}{2n}.
		\end{align*}
		Thus, $\nu\in\widetilde{V}$. Next, we shall show that $\widetilde{V}\subseteq\widetilde{\pi}_n(\widetilde{U})$.
		
		Now fix any $\tau\in\widetilde{V}\subseteq\mathcal{M}_n(Y)$. We have $\tau=\frac{1}{n}\sum\limits_{s=1}^{n}\delta_{u_{s}}$
		for some $u_{s}\in Y$. For every $i\in[h]$, we set 
		\begin{align}\label{Ap-6}
			L(i)\colon=\bigg\{s\in[n]\colon u_{s}\in \bigcap_{j=1}^{\ell_i}
			\pi\big(B_{\rho_X}(x_{i,j},r)\big)\bigg\}.
		\end{align}
		By $\tau\in\widetilde{V}$, one has
		\begin{align}\label{Ap-3}
			\vert L(i)\vert=n\cdot\tau\bigg(\bigcap_{j=1}^{\ell_i}
			\pi\big(B_{\rho_X}(x_{i,j},r)\big)\bigg)>k_i-\frac{1}{2},
		\end{align}
		for every $i\in[h]$. Since $\vert L(i)\vert\in\mathbb{N}$, by (\ref{Ap-3}), $\vert L(i)\vert\geq k_i$. 
		We note that, $L(i)$, $i\in[h]$ are pairwise disjoint since $\bigcap_{j=1}^{\ell_i}\pi\big(B_{\rho_X}(x_{i,j},r)\big)$, $i\in[h]$
		are pairwise disjoint. Moreover, by $\sum_{i\in[h]}k_i=n$, one has $\vert L(i)\vert=k_i$ for every $i\in[h]$. Hence,
		\begin{align}\label{Ap-9}
			\bigcup_{i\in[h]}L(i)=\bigsqcup_{i\in[h]} L(i)=[n]\ \ \text{and}\ \  \tau=\frac{1}{n}\sum_{i\in[h]}\big(\sum_{s\in L(i)}\delta_{u_s}\big).
		\end{align}
	
		For every $i\in[h]$, since $\vert L(i)\vert=k_i\overset{(a)}{=}\sum_{j\in[\ell_i]}m_{i,j}$, we can 
		rewrite $L(i)=\{s_1,s_2,\cdots,s_{k_i}\}$. For every $i\in[h]$ and $j\in[\ell_i]$, we denote 
		$R_i(j)=\sum_{t=1}^{j}m_{i,t}$ and $R_i(0)=0$.
		Then $R_i(\ell_i)=k_i$. By (\ref{Ap-6}), for every $j\in[\ell_i]$ and integer $q$ with
		$R_i(j-1)+1\leq q\leq R_i(j)$, there exists $x_{i,q}'\in B(x_{i,j},r)$ such that
		$\pi(x_{i,q}')=u_{s_q}$. Then by (\ref{Ap-8}) one has 
		$$\mu'\colon=\frac{1}{n}\sum_{i\in[h]}\sum_{j\in[\ell_i]}\sum_{q=R_i(j-1)+1}^{R(j)}
		\delta_{x_{i,q}'}\in\widetilde{U}$$
		and
		\begin{align*}
			\widetilde{\pi}_n(\mu')&=\frac{1}{n}\sum_{i\in[h]}\sum_{j\in[\ell_i]}\sum_{q=R_i(j-1)+1}^{R_i(j)}
			\delta_{u_{s_q}}=\frac{1}{n}\sum_{i\in[h]}\sum_{q\in[k_i]}\delta_{u_{s_q}}\\
			&=\frac{1}{n}\sum_{i\in[h]}\sum_{s\in L(i)}\delta_{u_s}\overset{(\ref{Ap-9})}{=}\tau.
		\end{align*}
		This implies $\widetilde{V}\subset\widetilde{\pi}_n(\widetilde{U})$. Hence, $\widetilde{\pi}_n(\widetilde{U})$ is 
		an open subset of $\mathcal{M}_n(Y)$ and $\widetilde{\pi}_n$ is open. 

	$(1)\Rightarrow (2)$ Now we assume that $\pi$ is open. Let 
	$\widetilde{U}$ be an open subset of $\mathcal{M}(X)$, we shall show $\widetilde{\pi}(\widetilde{U})$ is open in 
	$\mathcal{M}(Y)$. 
	
	For every $\nu\in\widetilde{\pi}(\widetilde{U})$, 
	there exists $\mu\in\widetilde{U}$ such that $\nu=\widetilde{\pi}(\mu)$. Next we shall show that there exists 
	$\delta>0$ small enough such that if we set
	$$\widetilde{V}:=\{\tau\in\mathcal{M}(Y): d_{P}(\nu,\tau)<\delta\},$$ 
	where
	$$\small{d_P(\tau,\nu):=inf\{\delta>0:\tau(A)\leq\nu(A^{\delta})+\delta\ \text{and}\ \nu(A)\leq\tau(A^{\delta})+\delta \ \text{for\ all}\ A\in\mathcal{B}_Y\}},$$
	then $\widetilde{V}$ is an open neighbourhood of $\nu$ contianed in $\widetilde{\pi}(\widetilde{U})$.
	
	Since $\widetilde{U}$ is open, by 
	Proposition \ref{basis of M(X)}, there exist $k\in\mathbb{N}$ and an open set of the form 
	$\mathbb{W}(U_1,U_2,\cdots,U_k;$ $\eta_1,\eta_2,\cdots,\eta_k)$ of $\mathcal{M}(X)$, where 
	$U_1,U_2,\cdots,U_k$ are disjoint non-empty open subsets of $X$ and $\eta_1,\eta_2,
	\cdots,\eta_k$ are positive real numbers with $\eta_1+\eta_2+
	\cdots+\eta_k<1$, such that
	$$\mu\in\mathbb{W}(U_1,U_2,\cdots,U_k;\eta_1,\eta_2,\cdots,\eta_k)\subset\overline{\mathbb{W}(U_1,U_2,
		\cdots,U_k;\eta_1,\eta_2,\cdots,\eta_k)}\subset\widetilde{U}.$$
	
	For any $t_1,t_2\in\{0,1\}^{[k]}$, we denote $t_1>t_2$ if $t_1\neq t_2$ and $t_1(i)\geq t_2(i)$ for every 
	$i\in[k]$.
	For every $\sigma\in\{0,1\}^{[k]}$, we set
	$$V_{\sigma}:=\bigcap_{\substack{i\in[k]\\ \sigma(i)=1}}\pi(U_i), \ \ 
	V'_{\sigma}\colon=V_{\sigma}\backslash\bigcup\limits_{\substack{\alpha\in\{0,1\}^{[k]}\\\alpha>\sigma}}V_{\alpha},$$
	and
	\begin{align}\label{mathcal-E}
		\mathcal{E}\colon=\big\{t\in\{0,1\}^{[k]}\colon\ \nu(V_t')>0\big\}.
	\end{align}
	Recall that for any subset $A$ of $Y$ and $a>0$, we denote $A^{a}=\{y\in Y\ \colon\rho_Y(y,A)<a\}$, where $\rho_Y$ is
	the compatible metric on $Y$. For every $i\in[k]$, since $\pi$ is open and $U_i$ is open in $X$, then 
	$\pi(U_i)$ is an open subset of $Y$. Then by 
	inner regularity, there exist $\varepsilon>0$ small enough, $\delta\in(0,\varepsilon)$ and compact subsets 
	$C_i$ of $Y$ for $i\in[k]$ such that
	\begin{enumerate}
		\item[(c1)] $\mu(U_i)>\eta_i+6^{k}\varepsilon$ for every $i\in[k]$,
		\item[(c2)] $\nu(V_{\sigma}')>5k\varepsilon$, for every $\sigma\in\mathcal{E}$,
		\item[(c3)] $C_i\subseteq C_i^{\delta}\subseteq C_i^{2\delta}\subseteq\pi(U_i)$ for every $i\in[k]$,
		\item[(c4)] $\nu(C_i)>\nu(\pi(U_i))-\varepsilon$ for $i\in[k]$. 
	\end{enumerate}
	
	Now we set
	$$\widetilde{V}:=\{\tau\in\mathcal{M}(Y): d_{P}(\nu,\tau)<\delta\}.$$ 
	Clearly, $\widetilde{V}$ is an open subset of $\mathcal{M}(Y)$ containing $\nu$. Now it is sufficient to prove
	that $\widetilde{V}\subset\widetilde{\pi}(\widetilde{U})$.

	For every $\sigma\in\{0,1\}^{[k]}$, we set 
	$$C_{\sigma}(\delta)\colon=\bigcap\limits_{\substack{i\in[k]\\\sigma(i)=1}}C_i^{\delta}\ \ \text{and}
	\ \ C_{\sigma}'\colon=
	C_{\sigma}(\delta)\backslash\bigcup_{\substack{\alpha\in\{0,1\}^{[k]}\\\alpha>\sigma}}C_{\alpha}(\delta).$$
	Then for every $\sigma\in\{0,1\}^{[k]}$, by (c3) and (c4) we have
	\begin{align}\label{pro4.1-1}
		\nu(V_{\sigma})&=\nu\bigg(\bigcap_{\substack{i\in[k]\\\sigma(i)=1}}\pi(U_i)\bigg)\nonumber\\
		&\overset{(c3)}{=}\nu\bigg(\bigcap_{\substack{i\in[k]\\\sigma(i)=1}}\bigg(\big(\pi(U_i)\backslash C_i\big)\cup C_i\bigg)\bigg)
		\overset{(c4)}{\leq}\nu(\bigcap_{\substack{i\in[k]\\\sigma(i)=1}}C_i)+k\varepsilon.
	\end{align}
	
	We note that, for any $t_1\neq t_2\in\{0,1\}^{[k]}$ one has $C_{t_1}'\cap C_{t_2}'=\emptyset$.
	In fact, if we define $t_1\vee t_2\in\{0,1\}^{[k]}$ by
	$$(t_1\vee t_2)(i)=\max\{t_1(i),t_2(i)\}$$
	for every $i\in[k]$, then it is clear that $t_1\vee t_2>t_1$ or $t_1\vee t_2>t_2$.
	Without loss of generality, we can assume $t_1\vee t_2>t_1$,
	then $C_{t_1}'\subseteq C_{t_1}(\delta)\backslash{C_{t_1\vee t_2}(\delta)}$. While,
	\begin{align*}
		C_{t_1}'\cap C_{t_2}'\subseteq C_{t_1}(\delta)\cap C_{t_2}(\delta)= C_{t_1\vee t_2}(\delta).
	\end{align*}
	Hence, $C_{t_1}'\cap C_{t_2}'=\emptyset$.
	
	Now for any fixed $\tau\in\widetilde{V}$, we shall show $\tau\in\widetilde{\pi}(\widetilde{U})$.
	By $d_P(\nu,\tau)<\delta$ one has
	$\tau(A^\delta)\geq \nu(A)-\delta$ for every $A\in \mathcal{B}_Y$. Then for 
	every $\sigma\in\mathcal{E}$, 
	\begin{small}
	\begin{align}\label{pro4.1-2}
		\tau\big(C_{\sigma}(\delta)\big)&=\tau\bigg(\bigcap_{\substack{i\in[k]\\\sigma(i)=1}}C_{i}^{\delta}
		\bigg)\geq \tau\bigg(\big(\bigcap_{\substack{i\in[k]\\\sigma(i)=1}}C_i\big)^{\delta}\bigg)\nonumber\\
		&\geq\nu\bigg(\bigcap_{\substack{i\in[k]\\\sigma(i)=1}}C_i\bigg)-\delta\overset{(\ref{pro4.1-1})}{\geq}\nu(V_{\sigma})-k\varepsilon-\delta.
	\end{align}
\end{small}
	Moreover, for every $\sigma\in\{0,1\}^{[k]}$,
	\begin{align}\label{pro4.1-3}
		&\bigg(\bigcup_{\substack{\alpha\in\{0,1\}^{[k]}\\\alpha>\sigma}}C_{\alpha}(\delta)\bigg)^{\delta}
		=\bigcup_{\substack{\alpha\in\{0,1\}^{[k]}\\\alpha>\sigma}}\big(C_{\alpha}(\delta)\big)^{\delta}
		=\bigcup_{\substack{\alpha\in\{0,1\}^{[k]}\\\alpha>\sigma}}
		\big(\bigcap_{\substack{i\in[k]\\\alpha(i)=1}}C_{i}^{\delta}\big)^{\delta}\nonumber\\
		&\subseteq\bigg(\bigcup_{\substack{\alpha\in\{0,1\}^{[k]}\\\alpha>\sigma}}\big(\bigcap_{\substack{i\in[k]\\\alpha(i)=1}}C_i^{2\delta}\big)\bigg)
		\overset{(c3)}{\subseteq}\bigg(\bigcup_{\substack{\alpha\in\{0,1\}^{[k]}\\\alpha>\sigma}}\big(\bigcap_{\substack{i\in[k]\\\alpha(i)=1}}\pi(U_i)\big)\bigg)\\
		&=\bigcup_{\substack{\alpha\in\{0,1\}^{[k]}\\\alpha>\sigma}}V_{\alpha}.\nonumber
	\end{align}
	Since $d_P(\nu,\tau)<\delta$ one has
	\begin{align}\label{d_p}
		\tau(A)\leq \nu(A^\delta)+\delta\ \text{for\ every}\ A\in\mathcal{B}_Y.
	\end{align}
	Note that for every $\alpha,\sigma\in\{0,1\}^{[k]}$ with $\alpha>\sigma$, one has $C_{\alpha}(\delta)\subseteq C_{\sigma}(\delta)$
	and $V_{\alpha}\subseteq V_{\sigma}$. Then for every $\sigma\in\mathcal{E}$,
	\begin{align}\label{pro4.1-4}
		\tau(C_{\sigma}')&=\tau(C_{\sigma}(\delta))-\tau\bigg(\bigcup_{\substack{\alpha\in\{0,1\}^{[k]}\\\alpha>\sigma}}C_{\alpha}(\delta)\bigg)\nonumber\\
		&\overset{(\ref{pro4.1-2})}{\geq}\nu(V_{\sigma})-k\varepsilon-\delta-\tau\bigg(\bigcup_{\substack{\alpha\in\{0,1\}^{[k]}\\\alpha>\sigma}}C_{\alpha}(\delta)\bigg)\nonumber\\
		&\overset{(\ref{d_p})}{\geq} \nu(V_{\sigma})-k\varepsilon-\delta-\nu\bigg(\big(\bigcup_{\substack{\alpha\in\{0,1\}^{[k]}\\\alpha>\sigma}}C_{\alpha}(\delta)\big)^\delta\bigg)-\delta\\
		&\overset{(\ref{pro4.1-3})}{\geq}\nu(V_{\sigma})-k\varepsilon-2\delta-\nu\bigg(\bigcup_{\substack{\alpha\in\{0,1\}^{[k]}\\\alpha>\sigma}}V_{\alpha}\bigg)\nonumber\\
		&=\nu(V_{\sigma}')-k\varepsilon-2\delta\geq \nu(V_{\sigma}')-3k\varepsilon>0.\nonumber
	\end{align}
	
	By $\overline{\bigcup_{n\in N}\mathcal{M}_n(Y)}=\mathcal{M}(Y)$ ,
	there exist $\tau_n=\frac{1}{n}\sum\limits_{j=1}^{n}\delta_{y_{n,j}}\in\mathcal{M}_n(Y)$ for $n\in\mathbb{N}$ and some $y_{n,j}\in Y$, $j\in[n]$, such that 
	$\tau_n\to\tau$ as $n\to\infty$. Moreover, since $\widetilde{V}$ is open in $\mathcal{M}(Y)$, we can find $N_0\in\mathbb{N}$ such that $\tau_n\in\widetilde{V}$ 
	for $n\geq N_0$.
	
	Let $n\geq N_0$. For every $\sigma\in\{0,1\}^{[k]}$ we set 
	$$S_\sigma^{n}\colon=\big\{h\in[n]\colon y_{n,h}\in C_{\sigma}'\big\}.$$
	Since $\tau_n\in \widetilde{V}$, by (\ref{pro4.1-4}) and recall that, for any $t_1\neq t_2\in\{0,1\}^{[k]}$ one has 
	$C_{t_1}'\cap C_{t_2}'=\emptyset$, then
	\begin{enumerate}
		\item[(i)] $S_{t_1}^n\cap S_{t_2}^n=\emptyset$, for any $t_1\neq t_2\in\{0,1\}^{[k]}$.
		\item[(ii)] $\vert S_{\sigma}^n\vert\geq n(\nu(V_{\sigma}')-3k\varepsilon)$ for every 
		$\sigma\in\mathcal{E}$, where $\mathcal{E}$ is defined as (\ref{mathcal-E}).
	\end{enumerate}
	Now, for every $\sigma\in\{0,1\}^{[k]}$ and $i\in[k]$, we set 
	\begin{align}\label{U_isigma}
		U_{i,\sigma}\colon=U_i\cap\pi^{-1}(V_{\sigma}')\ \ \text{and} \ \ a_{i,\sigma}\colon=\mu(U_{i,\sigma}).
	\end{align}
	
	Fix any $\sigma\in\{0,1\}^{[k]}$. We can rewrite $\{i\in[k]\colon \sigma(i)=1\}$ as $\{i_1<i_2<\cdots<i_{q}\}$
	for some $q\in\mathbb{N}$. For $i_1$, we choose arbitrarily a subset $S_{\sigma,i_1}^{n}$ of $S_{\sigma}^n$ with 
	$\vert S_{\sigma,i_1}^n\vert=\lfloor \frac{a_{i_1,\sigma}}{\sum\limits_{\ell\in[k],\sigma(\ell)=1}a_{\ell,\sigma}}\vert S_{\sigma}^n\vert\rfloor $, where we note: $\frac{0}{0}=0$. 
	For $i_2$, we choose arbitrarily a subset $S_{\sigma,i_2}^{n}$ of $S_{\sigma}^n\setminus S_{\sigma,i_1}^n$ with
	$\vert S_{\sigma,i_2}^n\vert=\lfloor \frac{a_{i_2,\sigma}}{\sum\limits_{\ell\in[k],\sigma(\ell)=1}a_{\ell,\sigma}}\vert S_{\sigma}^n\vert\rfloor$.
	We continue inductively obtaining 
	$$S_{\sigma,i_j}^n\subseteq S_{\sigma}^n\setminus\big(S_{\sigma,i_1}^n\cup S_{\sigma,i_2}^n\cdots\cup S_{\sigma,i_{j-1}}^n\big)$$
	for $j=3,4,\cdots,q-1$, with
	$\vert S_{\sigma,i_j}^n\vert=\lfloor \frac{a_{i_j,\sigma}}{\sum\limits_{\ell\in[k],\sigma(\ell)=1}a_{\ell,\sigma}}\vert S_{\sigma}^n\vert\rfloor$.
	We set $S_{\sigma,i_q}^n=S_{\sigma}^n\setminus(\bigcup_{j=1}^{q-1}S_{\sigma,i_j}^n)$.
	Additionally, we note that
	$$y_{n,h}\in C_{\sigma}'\subseteq\bigcap\limits_{\substack{i\in[k]\\\sigma(i)=1}}\pi(U_i)=\bigcap_{\ell=1}^{q} \pi(U_{i_\ell})$$
	for every $h\in S_{\sigma}^n$. Then we have the following properties for $S_{\sigma,i}^n$, $i\in[k]$.
	\begin{enumerate}
		\item[(i*)] $\vert S_{\sigma,i}^n\vert\geq\lfloor \frac{a_{i,\sigma}}{\sum\limits_{\ell\in[k],\sigma(\ell)=1}a_{\ell,\sigma}}\vert S_{\sigma}^n\vert
		\rfloor$ for every $i\in[k]$ with $\sigma(i)=1$.
		\item[(ii*)] For every $i\in[k]$ with $\sigma(i)=1$, if $h\in S_{\sigma,i}^n$ then there exists $x_{n,h}^{\sigma}\in U_i$ satisfying $\pi(x_{n,h}^{\sigma})=y_{n,h}$.
		\item[(iii*)]$S_{\sigma,i'}^n\cap S_{\sigma,i''}^n=\emptyset$ for every $i'\neq i''\in\{i\in[k]\colon\sigma(i)=1\}$ and $\bigcup\limits_{i\in[k],\sigma(i)=1}S_{\sigma,i}^n=S_{\sigma}^n$.
	\end{enumerate}
	Since $\pi$ is surjective, for every $h'\in S_0^n:=[n]\bigg\backslash\bigg(\bigcup\limits_{\sigma\in\{0,1\}^{[k]}}S_{\sigma}^n\bigg)$, 
	there exists $x_{n,h'}\in X$ such that $\pi(x_{n,h'})=y_{n,h'}$. Now we set
	\begin{align}\label{mu_n}
		\mu_n\colon&=\frac{1}{n}\big(\sum_{\sigma\in \{0,1\}^{[k]}}\sum_{h\in S_{\sigma}^n}\delta_{x_{n,h}^{\sigma}}+\sum_{h'\in S_0^n}\delta_{x_{n,h'}}\big)\nonumber\\
		&\overset{\text{(iii*)}}{=}\frac{1}{n}\big(\sum_{\sigma\in \{0,1\}^{[k]}}\sum_{\substack{i\in[k]\\ \sigma(i)=1}}\sum_{h\in S_{\sigma,i}^n}
		\delta_{x_{n,h}^{\sigma}}+\sum_{h'\in S_0^n}\delta_{x_{n,h'}}\big)
	\end{align}
	
	Clearly, $\widetilde{\pi}(\mu_n)=\tau_n$. We claim that
	$\mu_{n}(U_{i_0})>\eta_{i_0}$ for every 
	$i_0\in[k]$ when $n$ is sufficiently large. Once it is true, we have 
	$$\mu_n\in\mathbb{W}(U_1,U_2,\cdots,U_k;\eta_1,\eta_2,\cdots,\eta_k).$$
	Then we can finde a sequence $n_1<n_2<\cdots$ such that $\lim\limits_{i\to\infty}\mu_{n_i}=\mu'$ 
	for some $\mu'\in\mathcal{M}(X)$. 
	Thus
	$$\mu'\in\overline{\mathbb{W}(U_1,U_2,\cdots,U_k;\eta_1,\eta_2,\cdots,\eta_k)}\subset\widetilde{U}$$
	and $\widetilde{\pi}(\mu')=\lim\limits_{i\to\infty}\widetilde{\pi}(\mu_{n_i})=\lim\limits_{i\to\infty}\tau_{n_i}=\tau$. By the arbitrariness of $\tau$, one has $\widetilde{V}\subseteq \widetilde{\pi}(\widetilde{U})$. This will end our proof. 
	
	Now, we shall show the claim: 
	$\mu_{n}(U_{i_0})>\eta_{i_0}$ for every $i_0\in[k]$ when $n$ is sufficiently large. 
	To show that, for any fixed $i_0\in\{1,2,\cdots,k\}$, we need the following facts.
	
	\textbf{Fact 1}: $\sum\limits_{\sigma\in\mathcal{E},\sigma(i_0)=1}\mu\big(U_{i_0}\cap\pi^{-1}(V_{\sigma}')\big)=\mu\bigg(U_{i_0}\cap\bigcup\limits_{\sigma\in\mathcal{E},\sigma(i_0)=1}\pi^{-1}(V_{\sigma}')\bigg)$.
	
	In fact, for any $t_1\neq t_2\in\{0,1\}^{[k]}$, if $y\in V_{t_1}'\cap V_{t_2}'\subseteq V_{t_1}\cap V_{t_2}$, then 
	$y\in V_{t_1\vee t_2}$. Since $t_1\vee t_2>t_1\ \text{or}\ t_2$, one has $y\notin V_{t_1}'$ or $y\notin V_{t_2}'$, which is a contradiction of $y\in V_{t_1}'\cap V_{t_2}'$. Hence, 
	$V_{t_1}'\cap V_{t_2}'=\emptyset$. Then Fact 1 follows.
	
	\textbf{Fact 2}: $\nu\bigg(\bigg(\bigcup\limits_{\substack{\sigma\in\mathcal{E}\\\sigma(i_0)=1}}V_{\sigma}'\bigg)\Delta
	\bigg(\bigcup\limits_{\substack{\sigma\in\{0,1\}^{[k]}\\\sigma(i_0)=1}}V_{\sigma}\bigg)\bigg)=0$, where $A\Delta B$ denotes $(A\setminus B)\cup(B\setminus A)$ for every $A,B\in\mathcal{B}_Y$.
	
	In fact, it is clear that 
	\begin{align}\label{ttt}
		\bigcup\limits_{\substack{\sigma\in\mathcal{E}\\\sigma(i_0)=1}}V_{\sigma}'\subseteq
		\bigcup\limits_{\substack{\sigma\in\{0,1\}^{[k]}\\\sigma(i_0)=1}}V_{\sigma}
	\end{align}
	Since $\sigma\in\{0,1\}^{[k]}\backslash\mathcal{E}$ implies $\nu(V_{\sigma}')=0$, one has
	\begin{align}\label{e11111}
		\nu\bigg(\bigcup\limits_{\substack{\sigma\in\mathcal{E}\\\sigma(i_0)=1}}V_{\sigma}'\bigg)
		=\nu\bigg(\bigcup\limits_{\substack{\sigma\in\{0,1\}^{[k]}\\\sigma(i_0)=1}}V_{\sigma}'\bigg).
	\end{align}
	Clearly, $\bigcup\limits_{\substack{\sigma\in\{0,1\}^{[k]}\\\sigma(i_0)=1}}V_{\sigma}\supseteq \bigcup\limits_{\substack{\sigma\in\{0,1\}^{[k]}\\\sigma(i_0)=1}}
	V_{\sigma}'$. Moreover, for 
	any given $x\in\bigcup\limits_{\substack{\sigma\in\{0,1\}^{[k]}\\\sigma(i_0)=1}}V_{\sigma}$, if we define $\sigma'$ as 
	$$\sigma'(i)=\max\big\{t(i)\colon t\in\{0,1\}^{[k]}\ \text{with} \ x\in V_{t}\big\}$$
	for every $i\in[k]$, then 
	$\sigma'(i_0)=1$ and $x\in V_{\sigma'}^{'}\subseteq \bigcup\limits_{\substack{\sigma\in\{0,1\}^{[k]}\\\sigma(i_0)=1}}
	V_{\sigma}'$. Hence 
	\begin{align}\label{e2222}
		\bigcup\limits_{\substack{\sigma\in\{0,1\}^{[k]}\\\sigma(i_0)=1}}V_{\sigma}=\bigcup\limits_{\substack{\sigma\in\{0,1\}^{[k]}\\\sigma(i_0)=1}}
		V_{\sigma}'.
	\end{align} 
	By (\ref{ttt}), (\ref{e11111}) and (\ref{e2222}), Fact 2 holds.
	
	\textbf{Fact 3}: For every $\sigma\in\mathcal{E}$, $\sum\limits_{\ell\in[k],\sigma(\ell)=1}a_{\ell,\sigma}\leq\nu(V_{\sigma}')$.
	
	Note that, $U_1,U_2,\cdots,U_k$ are disjoint. Then by (\ref{U_isigma}) we have
	\begin{align*}
		&\sum_{\ell\in[k],\sigma(\ell)=1}a_{\ell,\sigma}\overset{(\ref{U_isigma})}{=}
		\sum_{\ell\in[k],\sigma(\ell)=1}\mu(U_{\ell,\sigma})\\
		\overset{(\ref{U_isigma})}{=}&\sum_{\ell\in[k],\sigma(\ell)=1}\mu(U_{\ell}\cap\pi^{-1}(V_{\sigma}'))
		=\mu\bigg(\big(\bigcup_{\ell\in[k],\sigma(\ell)=1}U_{\ell}\big)\cap \pi^{-1}(V_{\sigma}')\bigg)\\
		\leq&\mu(\pi^{-1}(V_{\sigma}'))=\nu(V_{\sigma}').
	\end{align*}
	Thus Fact 3 holds.
	
	Now by Fact 1-3, we have
	\begin{align}\label{pro4.1-5}
		\sum_{\sigma\in\mathcal{E},\sigma(i_0)=1}a_{i_0,\sigma}&\overset{(\ref{U_isigma})}{=}\sum_{\sigma\in\mathcal{E},\sigma(i_0)=1}\mu(U_{i_0}\cap\pi^{-1}V_{\sigma}')\nonumber\\
		&\overset{\text{(Fact\ 1)}}{=}\mu\bigg(U_{i_0}\cap\bigcup\limits_{\sigma\in\mathcal{E},\sigma(i_0)=1}\pi^{-1}(V_{\sigma}')\bigg)\\
		&\overset{\text{(Fact\ 2)}}{=}\mu\bigg(U_{i_0}\cap\pi^{-1}\big(\bigcup\limits_{\substack{\sigma\in\{0,1\}^{[k]}\\\sigma(i_0)=1}}V_{\sigma}\big)\bigg).\nonumber
	\end{align}
	We define $t\in\{0,1\}^{[k]}$ as $t(i_0)=1$ and $t(i)=0$ for each $i\in[k]\backslash\{i_0\}$. Then 
	$\pi(U_{i_0})=V_{t}\subseteq \bigcup\limits_{\substack{\sigma\in\{0,1\}^{[k]}\\\sigma(i_0)=1}}V_{\sigma}$. Thus
	$U_{i_0}\subseteq\pi^{-1}\big(\bigcup\limits_{\substack{\sigma\in\{0,1\}^{[k]}\\\sigma(i_0)=1}}V_{\sigma}\big)$
	and by (\ref{pro4.1-5}) we have
	\begin{align}\label{pro4.1-6}
		\sum_{\sigma\in\mathcal{E},\sigma(i_0)=1}a_{i_0,\sigma}=\mu(U_{i_0}).
	\end{align}
	
	Then for any $n>N_0$, we have 
	\begin{align*}
		\mu_n(U_{i_0})&\overset{(\ref{mu_n})}{\geq}\frac{1}{n}\sum_{\sigma\in\{0,1\}^{[k]}}\sum_{\substack{i\in[k]\\ \sigma(i)=1}}\sum_{h\in S_{\sigma,i}^n}\delta_{x_{n,h}^{\sigma}}(U_{i_0})
		\geq\frac{1}{n}\sum_{\sigma\in\mathcal{E}}\sum_{\substack{i\in[k]\\ \sigma(i)=1}}\sum_{h\in S_{\sigma,i}^n}\delta_{x_{n,h}^{\sigma}}(U_{i_0})\\
		&\overset{\text{(ii*)}}{\geq}\frac{1}{n}\sum_{\substack{\sigma\in\mathcal{E}\\ \sigma(i_0)=1}}\vert S_{\sigma,i_0}^n\vert
		\overset{\text{(i*)}}{\geq}\frac{1}{n}\sum_{\substack{\sigma\in\mathcal{E}\\ \sigma(i_0)=1}}\bigg\lfloor \frac{a_{i_0,\sigma}}{\sum\limits_{\ell\in[k],\sigma(\ell)=1}a_{\ell,\sigma}}\vert S_{\sigma}^n\vert\bigg\rfloor\\
		&\geq \frac{1}{n}\bigg(\sum_{\substack{\sigma\in\mathcal{E}\\ \sigma(i_0)=1}}\bigg(\frac{a_{i_0,\sigma}}{\sum\limits_{\ell\in[k],\sigma(\ell)=1}a_{\ell,\sigma}}\vert S_{\sigma}^n\vert\bigg)\bigg)-\frac{2^k}{n}\\
		&\overset{\text{(ii)}}{\geq}\frac{1}{n}\bigg(\sum_{\substack{\sigma\in\mathcal{E}\\ \sigma(i_0)=1}}\bigg(\frac{a_{i_0,\sigma}}{\sum\limits_{\ell\in[k],\sigma(\ell)=1}a_{\ell,\sigma}}n\big(\nu(V_{\sigma}')-3k\varepsilon\big)\bigg)\bigg)-\frac{2^k}{n}\\
		&\geq\sum_{\substack{\sigma\in\mathcal{E}\\\sigma(i_0)=1}}\bigg(\frac{a_{i_0,\sigma}}{\sum\limits_{\ell\in[k],\sigma(\ell)=1}a_{\ell,\sigma}}\nu(V_{\sigma}')\bigg)-2^k\cdot 3k\varepsilon-\frac{2^k}{n}\\
		&\overset{\text{(Fact\ 3)}}{\geq}\sum_{\substack{\sigma\in\mathcal{E}\\\sigma(i_0)=1}}a_{i_0,\sigma}-2^k\cdot 3k\varepsilon-\frac{2^k}{n}
		\overset{(\ref{pro4.1-6})}{=}\mu(U_{i_0})-2^k\cdot 3k\varepsilon-\frac{2^k}{n}.
	\end{align*}
	Then by letting $n\to\infty$, for every $i_0\in[k]$ since $\mu(U_{i_0})>\eta_{i_0}+6^{k}\varepsilon$ by (c1), we have $\mu_n(U_{i_0})>\eta_{i_0}$. This ends the proof of the claim. 
	\end{proof}
\begin{appendix}
	\section{proof of Theorem \ref{product is upe}}\label{B}
	Let $\pi\colon (X,G)\to (Y,G)$ be a factor map between two $G$-systems. For any $n\in\mathbb{N}$ and a tuple
	$\mathcal{V}=(V_1,V_2,\cdots,V_n)$ of subsets of $X$, recall that we denote by $\mathcal{P}_{\mathcal{V}}^{\pi}$ the set of all 
	independence sets of $\mathcal{V}$ with respect to $\pi$. 
	
	Identifying subsets of $G$ with elements of 
	$\{0,1\}^G$ by taking indicator functions, we may think of $\mathcal{P}_{\mathcal{V}}^{\pi}$ as a subset 
	of $\{0,1\}^G$. Endow $\{0,1\}^G$ with the shift given by 
	$$(s\sigma)(t)=\sigma(ts)$$
	for all $\sigma\in\{0,1\}^G$ and $s,t\in G$. It is clear that $\mathcal{P}_{\mathcal{V}}^{\pi}$ is shift-invariant.
	Moreover, when $V_1,V_2,\cdots,V_n$ are closed subsets of 
	$X$, $\mathcal{P}_{\mathcal{V}}^{\pi}$ is also closed in $\{0,1\}^G$.
	
	We say a closed and shift-invariant subset $\mathcal{P}\subseteq\{0,1\}^G$ \textbf{has
		positive density} if there exists constant $c>0$ such that for every nonempty subset $F$ of $G$, there exists 
		$I\in\mathcal{P}$ with $I\subseteq F$, such that $\vert I\vert>c\vert F\vert$. 
		Then by Corollary \ref{rel-entropy-pair}, we immediately have the following property.
	\begin{prop}\label{B_0}Let $\pi\colon (X,G)\to(Y,G)$ be a factor map between two $G$-systems, $(x_1,x_2)\in X\times X\backslash\Delta(X)$.		
		Then $(x_1,x_2)\in E(X,G|\pi)$ if and only if for any disjoint open subsets $V_1, V_2$ of $X$ with $x_i\in V_i$ for $i=1,2$, 
		$\mathcal{P}_{\{V_1,V_2\}}^{\pi}$ has positive density.
	\end{prop}
	
	Let $e_G$ be the unit of $G$, we set 
	$$[e_G]\colon=\big\{\omega\in\{0,1\}^G\colon \omega(e_G)=1\big\}.$$ 
	Then we have the follwing lemmas.
	\begin{lem}\cite[Lemma 12.6]{Kerr-Li-2016}\label{B_1} Let $A$ be a closed subset of $X$. Then 
		$\mathcal{P}_A\colon=\big\{I\subseteq G\colon \bigcap_{g\in I}g^{-1}A\neq\emptyset\big\}$ has positive density if and only if there exists
		$\mu\in\mathcal{M}(X,G)$ with $\mu(A)>0$. 
	\end{lem}
	The following lemma is proved when $G=\mathbb{Z}$ in \cite[Proposition 3.9]{Huang-Ye-Zhang-1}. We omit the proof.
	\begin{lem}\label{prop 0} Let $\pi: (X,G)\to (Z,G)$, $\pi_1: (X,G)\to (Y,G)$ and $\pi_2: (Y,G)\to (Z,G)$ be three factor maps such that $\pi=\pi_2\cdot \pi_1$. 
		Then $\pi$ has rel-u.p.e implies $\pi_2$ has rel-u.p.e.
	\end{lem} 

	For a factor map $\pi\colon (X,\mathbb{Z})\to(Y,\mathbb{Z})$ between two $\mathbb{Z}$-systems, the authors in \cite{Huang-Ye-Zhang-1} proved that 
	if $\pi$ has rel-u.p.e, then $supp(Y)=Y$ implies $supp(X)=X$ (see \cite[Theorem 5.4]{Huang-Ye-Zhang-1}). For discrete countable amenable group $G$, we have 
	the same result. 
	\begin{prop}\label{supp} 
		Let $\pi\colon(X,G)\to(Y,G)$ be a factor map between two $G$-systems. If $\pi$ has rel-u.p.e and $supp(Y)=Y$, then $supp(X)=X$.
	\end{prop} 
	\begin{proof}
		Assume that $supp(X)\not=X$, then there exist $x_1\in X$ and an open neighbourhood $V$ of $x_1$ such that $V\cap supp(X)=\emptyset$.
		Let $U=\bigcup_{g\in G}g^{-1}V$, then $U$ is open and $\mu(U)=0$ for every $\mu\in\mathcal{M}(X,G)$. 
		Thus $supp(X)\subseteq U^c$, where $U^c=X\setminus U$.

		Let $y=\pi(x_1)$. We note that, $\pi^{-1}\{y\}\cap U^c\neq \emptyset$. In fact,
		since $supp(Y)=Y$ there exits $\nu\in\mathcal{M}(Y,G)$ such that $y\in supp(\nu)$.
		Then there exists $\tilde{\mu}\in\mathcal{M}(X,G)$ such that $\widetilde{\pi}(\tilde{\mu})=\nu$. 
		If $\pi^{-1}\{y\}\subseteq U$, there exists $\delta>0$ such that 
        $\pi^{-1}B(y,\delta)\subseteq U$. Then $\nu(B(y,\delta))=\tilde{\mu}(\pi^{-1}B(y,\delta))=0$. This contradicts
		$y\in supp(\nu)$. Thus there exists $x_2\in U^c$ such that $\pi(x_2)=y$.
		
		By Urysohn's lemma, there exists continuous function $f: X\to[0,1]$ such that $f(x_1)=0$ and $f(x)=1$ for any $x\in U^{c}$. 
		We set 
		$$F\colon X\to [0,1]^G \ \ \text{by}\ \  (F(x))(g)=f(gx).$$
		Consider the $G$-action on $[0,1]^G$ defined by $(g\omega)(h)=\omega(hg)$ for every $\omega\in[0,1]^G$ and 
		$g,h\in G$. We define a factor map
		$$\phi\colon (X,G)\to ([0,1]^G\times Y,G), \ \ \text{by}\ \ \phi(x)=(F(x),\pi(x)).$$
		Let $W=\phi(X)$, and $\pi_2\colon (W,G)\to (Y,G)$ be the projection map to the second coordinate. Then $\pi=\pi_2\circ \phi$.
		By Proposition \ref{prop 0}, $\pi_2$ has rel-u.p.e.
        Note that $\pi_2(\phi(x_1))=\pi(x_1)=\pi(x_2)=\pi_2(\phi(x_2))$ and $\phi(x_1)\neq\phi(x_2)$. Thus
		$$(\phi(x_1),\phi(x_2))\in R_{\pi_2}\setminus \Delta(W)=E(W,G\vert \pi_2).$$
		Then, by Lemma \ref{B_1}, one has $\phi(x_1)\in supp(W)$. While, $\phi(x_1)\notin \{1^G\}\times Y$ and
		for every $\mu\in\mathcal{M}(X,G)$ one has $supp(\mu)\subseteq U^c$, which implies
		$supp(W)\subseteq\phi(U^c)\subseteq \{1^G\}\times Y$. Thus $\phi(x_1)\notin supp(W)$. It is 
		a contradiction.
	\end{proof}

	Now we are ready to give the proof of Theorem \ref{product is upe}.
	\begin{thm} \label{Bpthm}
		Let $\pi_i\colon (X_i,G)\to (Y_i,G)$ be two factor maps between $G$-systems 
		and $supp(Y_i)=Y_i$ for $i=1,2$. Then $\pi_1$ and $\pi_2$ has rel-u.p.e if and only if 
		$\pi_1\times \pi_2\colon (X_1\times X_2,G)\to (Y_1\times Y_2,G)$ has rel-u.p.e.
	\end{thm}
	\begin{proof}For the nontrivial direction, if $\pi_1$ and $\pi_2$ have rel-u.p.e, for any 
		$u_1=(x_1,z_1)$ and $u_2=(x_2,z_2)$ in $ X_1\times X_2$ with $(u_1,u_2)\in R_{\pi_1\times\pi_2}\setminus\Delta(X_1\times X_2)$,
		we shall prove $(u_1,u_2)\in E(X_1\times X_2,G\vert\pi_1\times\pi_2)$. 
		Without loss of generality, we assume $x_1\neq x_2$. 

		Let $\widetilde{U}_1=U_1\times V_1,\widetilde{U}_2=U_2\times V_2$ be neighbourhoods of $u_1$ and $u_2$ respectively.
		Note that,
		$(x_1,x_2)\in R_{\pi_1}\setminus\Delta(X_1)=E(X_1,G\vert\pi_1)$ since $\pi_1$ has rel-u.p.e. Then by Corollary \ref{rel-entropy-pair} there exists
		$c_1>0$ such that, for every $F\in\mathcal{F}(G)$ there exists $E\subseteq F$ with $\vert E\vert>c_1\vert F\vert$, which is an
		independence set of $\{U_1,U_2\}$ with respect to $\pi_1$. For $z_1$ and $z_2$ there are two cases.
		
		\textbf{Case 1:} $z_1\neq z_2$. In this case, $(z_1,z_2)\in R_{\pi_2}\setminus\Delta(X_2)=E(X_2,G\vert\pi_2)$ since $\pi_2$ has
		rel-u.p.e. Then there exists $c_2>0$ such that for every $F\in \mathcal{F}(G)$ there exists $F_0\subseteq F$
		with $\vert F_0\vert>c_1\cdot c_2\vert F\vert$, which is an independence set of $\{\widetilde{U}_1,\widetilde{U}_2\}$ with 
		respect to $\pi_1\times\pi_2$. This implies $(u_1,u_2)\in E(X_1\times X_2,G\vert\pi_1\times\pi_2)$.

		\textbf{Case 2:} $z_1=z_2=z$ for some $z\in X_2$. We set $V=V_1\cap V_2$. Then $V$ is an open neighbourhood of $z$. 
		Since $supp(Y_2)=Y_2$ and $\pi_2$ has rel-u.p.e, by Proposition \ref{supp}, we have $supp(X_2)=X_2$. Thus there exists $\nu\in\mathcal{M}(X_2,G)$,
		such that $\nu(V)>0$. By Lemma \ref{B_1}, $\mathcal{P}_V^{\pi_2}$ has positive density. Then by similar analysis in Case 1, we can also
		obtain that $(u_1,u_2)\in E(X_1\times X_2,G\vert\pi_1\times\pi_2)$. This ends our proof.
	\end{proof}

\end{appendix}

\bibliographystyle{amsplain}

\end{document}